\documentclass[11pt,a4paper]{amsart}		

\usepackage{hyperref}
\usepackage{xcolor}
\hypersetup{
    colorlinks,
    linkcolor={red!50!black},
    citecolor={blue!50!black},
    urlcolor={blue!80!black}
}

\usepackage[utf8]{inputenc}					

\usepackage{lmodern}						

\usepackage{amsmath,amsfonts,amssymb,amsthm,pifont}
\usepackage{amsbsy,amscd,mathrsfs,graphicx}
\usepackage{hyperref}

\usepackage{pdflscape}
\usepackage{float}
\usepackage{array}

\usepackage{tikz}
\usetikzlibrary{shapes,arrows}
	
\usepackage[left=3cm,right=3cm,top=3cm,bottom=3cm]{geometry}

\usepackage[makeroom]{cancel}				
\usepackage{mathtools}						
\mathtoolsset{showonlyrefs,showmanualtags}	

\theoremstyle{plain}
\newtheorem{theorem}{Theorem}
\newtheorem{prop}[theorem]{Proposition}
\newtheorem{cor}[theorem]{Corollary}
\newtheorem{lemma}[theorem]{Lemma}
\newtheorem*{lemma*}{Lemma}
\theoremstyle{definition}
\newtheorem{definition}[theorem]{Definition}
\theoremstyle{remark}
\newtheorem{rmk}{Remark}
\newtheorem{example}{Example}

\DeclareMathOperator{\spn}{\mathrm{span}}
\DeclareMathOperator{\rank}{\mathrm{rank}}
\DeclareMathOperator{\tr}{\mathrm{Tr}}
\DeclareMathOperator{\trace}{\mathrm{Tr}}
\let\div\relax\DeclareMathOperator{\div}{\mathrm{div}}
\DeclareMathOperator{\dive}{\mathrm{div}}
\DeclareMathOperator{\grad}{\mathrm{grad}}

\newcommand{\R}{\mathbb{R}}						
\newcommand{\distr}{\blacktriangle}				
\newcommand{\metr}{\mathbf{g}}					
\newcommand{\g}{\mathbf{g}}						

\newcommand{\mS}{\mathbb{S}}

\newcommand{\dt}{\delta}

\newcommand{\bqn}{\begin{eqnarray}}
\newcommand{\eqn}{\end{eqnarray}}
\newcommand{\bi}{\begin{itemize}}

\newcommand{\ei}{\end{itemize}}
\newcommand{\bc}{\begin{cor}}
\newcommand{\ec}{\end{cor}}
\newcommand{\bp}{\begin{prop}}
\newcommand{\ep}{\end{prop}}
\newcommand{\bt}{\begin{theorem}}
\newcommand{\et}{\end{theorem}}

\newcommand{\eps}{\varepsilon}

\newcommand{\cyl}{
\begin{tikzpicture}
\tikzset{every node/.style={cylinder, shape border rotate=90, draw}}
\node [minimum height=9pt, minimum width = 9pt, aspect =.3, scale=0.8, very thin] at (0,0) {};
\end{tikzpicture}
}
\newcommand{\cylsmall}{
\begin{tikzpicture}
\tikzset{every node/.style={cylinder, shape border rotate=90, draw}}
\node [minimum height=9pt, minimum width = 9pt, aspect =.3, scale=0.6, very thin] at (0,0) {};
\end{tikzpicture}
}

\newcommand{\ip}[2]{\left\langle#1,#2\right\rangle}
\newcommand{\lp}{\left(}
\newcommand{\rp}{\right)}
\newcommand{\lb}{\left[}
\newcommand{\rb}{\right]}

\newcommand{\E}{\mathbb{E}}
\newcommand{\bR}{\mathbb{R}}

\DeclareMathOperator{\sinc}{\mathrm{sinc}}						

\title[Intrinsic random walks via volume sampling]{Intrinsic random walks in Riemannian and sub-Riemannian geometry via volume sampling}

\author{Andrei Agrachev}
\author{Ugo Boscain}
\author{Robert Neel}
\author{Luca Rizzi}

\address[Andrei Agrachev]{SISSA, Trieste \& Steklov Math. Inst., Moscow.}
\email{\href{mailto:agrachev@sissa.it}{agrachev@sissa.it}}

\address[Ugo Boscain]{Inria, team GECO \& CMAP, \'Ecole Polytechnique, CNRS, Universit\'e Paris-Saclay, Palaiseau, France}
\email{\href{mailto:ugo.boscain@polytechnique.edu}{ugo.boscain@polytechnique.edu}}

\address[Robert Neel]{Department of Mathematics, Lehigh University, Bethlehem, PA, USA}
\email{\href{mailto:robert.neel@lehigh.edu}{robert.neel@lehigh.edu}}

\address[Luca Rizzi]{Univ. Grenoble Alpes, CNRS, Institut Fourier, F-38000 Grenoble, France}
\email{\href{mailto:luca.rizzi@univ-grenoble-alpes.fr}{luca.rizzi@univ-grenoble-alpes.fr}}

\date{\today}


\begin{document}

\begin{abstract}
We relate some constructions of stochastic analysis to differential geometry, via random walk approximations. We consider walks on both Riemannian and sub-Riemannian manifolds in which the steps consist of travel along either geodesics or integral curves associated to orthonormal frames, and we give particular attention to walks where the choice of step is influenced by a volume on the manifold. A primary motivation is to explore how one can pass, in the parabolic scaling limit, from geodesics, orthonormal frames, and/or volumes to diffusions, and hence their infinitesimal generators, on sub-Riemannian manifolds, which is interesting in light of the fact that there is no completely canonical notion of sub-Laplacian on a general sub-Riemannian manifold. However, even in the Riemannian case, this random walk approach illuminates the geometric significance of Ito and Stratonovich stochastic differential equations as well as the role played by the volume.
\end{abstract}

\maketitle

\tableofcontents

\section{Introduction}

Consider a Riemannian or a sub-Riemannian manifold $M$ and assume that $\{X_1,\ldots,X_k\}$ is a global orthonormal frame. It is well known that, under mild hypotheses, the solution $q_t$ to the stochastic differential equation in Stratonovich sense
\begin{equation}
dq_t=\sum_{i=1}^k X_i(q_t) \circ \lp \sqrt{2}\, dw^i_t\rp
\end{equation}
produces a solution to the heat-like equation
\begin{equation} \label{eq-strat}
\partial_t\varphi=\sum_{i=1}^k X_i^2 \varphi
\end{equation}
by taking $\varphi_t(q) = \E\lb \varphi_0(q_t)| q_0=q\rb$, where $\varphi_0$ gives the initial condition. (Here the driving processes $w^i_t$ are independent real Brownian motions, and $\sqrt{2}$ factor is there so that the resulting sum-of-squares operator doesn't need a $1/2$, consistent with the convention favored by analysts.) One can interpret \eqref{eq-strat} as the equation satisfied by a random walk with parabolic scaling  following the integral curves of the vector fields $X_1,\ldots X_m$, when the step of the walk tends to zero. This construction is very general (works in Riemannian and in the sub-Riemannian case) and does not use any notion of volume on the manifold.\footnote{In the Riemannian case avoiding the use of a volume is not crucial since an intrinsic volume (the Riemannian one) can always be defined. But in the sub-Riemannian case, how to define an intrinsic volume is a subtle question, as discussed below.}

However the operator $\sum_{i=1}^k X_i^2$ is not completely satisfactory to describe a diffusion process for the following reasons:

\begin{itemize}
\item the construction works only if a global orthonormal frame $X_1,\ldots,X_k$ exists;
\item it is not intrinsic in the sense that it depends on the choice of orthonormal frame;
\item it is not essentially self-adjoint w.r.t. a natural volume and one cannot guarantee a priori a ``good'' evolution in $L^2$ (existence and uniqueness of a contraction semigroup, etc...).
\end{itemize}
In the Riemannian context a heat operator that is globally well defined, frame independent and essentially self-adjoint w.r.t.\ the Riemannian volume (at least under the hypotheses of completeness) is the Laplace-Beltrami operator $\Delta=\div\circ\grad$. The heat equation
\begin{equation}
\partial_t \varphi=\Delta \varphi
\end{equation}
has an associated diffusion, namely Brownian motion (with a time change by a factor of 2), given by the solution of the stochastic differential equation
\begin{equation}
dq_t=\sum_{i=1}^k X_i(q_t) \lp\sqrt{2}\, dw^i_t\rp \qquad \text{(in this case $k=n$ is equal to the dimension of $M$)}
\end{equation}
in the Ito sense (for instance using the Bismut construction on the bundle of orthonormal frames \cite{bismut} or Emery's approach \cite{emery}).
Also, this equation can be interpreted  as the equation satisfied by the limit of a random walk that,  instead of integral curves, follows geodesics.  The geodesics starting from a point are weighted with a uniform probability given by the Riemannian metric on the tangent space at the point.

The purpose of this paper is to extend this more invariant construction of random walks to the sub-Riemannian context, to obtain a definition of an intrinsic Laplacian in sub-Riemannian geometry and to compare it with the divergence of the horizontal gradient.

The task of determining the appropriate random walk is not obvious for several reasons. First, in sub-Riemannian geometry geodesics starting from a given point are always parameterized by a non-compact subset of the cotangent space at the point, on which there is no canonical ``uniform'' probability measure. Second, in sub-Riemannian geometry for every $\eps$ there exist geodesics of length $\eps$ that have already lost their optimality, and one has to choose between a construction involving all geodesics (including the non-optimal ones) or only those that are optimal. Third, one should decide what to do with abnormal extremals.  Finally, there is the problem of defining an intrinsic volume in sub-Riemannian geometry, to compute the divergence.

It is not the first time that this problem has been attacked. In \cite{GordLae,OurUrPaper,Grong1,Grong2}, the authors compare the divergence of the gradient with the Laplacian corresponding to a random walk induced by a splitting of the cotangent bundle (see \cite[Section 1.4]{OurUrPaper} for a detailed summary of this literature). In this paper we take another approach, trying to induce a measure on the space of geodesics from the ambient space by ``sampling'' the volume at a point a fraction $c$ of the way along the geodesic, see Section~\ref{s:riem-RW}. In a broader context, discrete approximations to diffusions have a long history, with, for example, Wong-Zakai approximations being widespread. The present paper follows in a related tradition, going back to \cite{PinksyRiem}, of developing geometrically meaningful approximations to diffusions on Riemannian or sub-Riemannian manifolds, in part with the aim of elucidating the connection between the diffusion and more fundamental geometric features of the manifold and/or the dynamics of which the diffusion is an idealization. This direction has seen a fair amount of activity lately; besides the papers on random walks arising from splittings in sub-Riemannian geometry listed above, we mention the kinetic Brownian motion of \cite{IsmaelKinetic} (which gives a family of $C^1$ approximations to Riemannian Brownian motion with random velocity vector) and the homogenization of perturbations of the geodesic flow discussed in \cite{XueMei} (which also gives an approximation to Riemannian Brownian motion).

This idea works very well in the Riemannian case, permitting a random walk interpretation of the divergence of the gradient also when the divergence is computed w.r.t.\ an arbitrary volume. More precisely, the limiting diffusion is generated by the divergence of the gradient if and only if at least one of the two conditions are satisfied: (i) one is using the Riemannian volume; (ii) the parameter $c$ used to realize the ``volume sampling'' is equal to $1/2$, evoking reminiscences of the Stratonovich integral. From these results one also recognizes a particular role played by the Riemannian volume (see Section~\ref{s:riem-RW} and Corollary~\ref{c:self-adj-Riem}). (In this Riemannian case, $c=0$ corresponds to no sampling of the volume, and the limiting diffusion is just Brownian motion as above.)

In the sub-Riemannian case the picture appears richer and more complicated. Even for contact Carnot groups (see Section~\ref{s:SR-RW}) the volume sampling procedure is non-trivial, as one requires an explicit knowledge of the exponential map. For Heisenberg groups, one gets a result essential identical to the Riemannian case, i.e. that the limiting diffusion is generated by the divergence of the horizontal gradient if and only if at least one of the following is satisfied: (i) one is using the Popp volume; (ii) the parameter $c=1/2$. For general contact Carnot groups, the results are more surprising, since the generator of the limiting diffusion is not the expected divergence of the horizontal gradient (even the second-order terms are not the expected ones); however, the generator will be the divergence of the horizontal gradient with respect to a different metric on the same distribution, as shown in Section \ref{Sect:IntrinsicFormula}. Moreover, the result just described applies to two somewhat different notions of a geodesic random walk, one in which we walk along all geodesics, and one in which we walk only along minimizing ones. An important difference between these two approaches is that only the walk along minimizing geodesics gives a non-zero operator in the limit as the volume sampling parameter $c$ goes to 0 (see Section \ref{s:altconstr}). Moreover, this non-zero limiting operator turns out to be independent of the volume, so that it becomes a possible, if slightly unusual (the principle symbol is not the obvious one), candidate for a choice of intrinsic sub-Laplacian. This may be an interesting direction to explore.

Motivated by these unexpected results and difficulties in manipulating the exponential map in more general sub-Riemannian cases, in Section \ref{s:RW-flow} we try another construction in the general contact case (that we call the flow random walk with volume sampling), inspired by the classical Stratonovich integration and also including a volume sampling procedure. This construction, a priori not intrinsic (it depends on a choice of vector fields), gives rise in the limit to an intrinsic operator showing the particular role played by the Popp volume. This construction also gives some interesting hints in the Riemannian case;
unfortunately it cannot be easily generalized to situations of higher step or corank.

On the stochastic side, in Section~\ref{s:convergence}, we introduce a general scheme for the convergence of random walks of a sufficiently general class to include all our constructions, based on the results of \cite{OurUrPaper}. Further, in the process of developing the random walks just described, we naturally obtain an intuitively appealing description of the solution to a Stratonovich SDE on a manifold as a randomized flow along the vector fields $V_1,\ldots,V_k$ (which determine the integrand) while the solution to an Ito SDE is a randomized geodesic tangent to the vector fields $V_1,\ldots,V_k$ (as already outlined above for an orthonormal frame). This difference corresponds to the infinitesimal generator being based on second Lie derivatives versus second covariant derivatives. Of course, such an approximation procedure by random walks yields nothing about the diffusions solving these SDEs that is not contained in standard stochastic calculus, but the explicit connection to important geometric objects seems compelling and something that has not been succinctly described before, to the best of our knowledge. Further, it is then natural to round out this perspective on the basic objects of stochastic calculus on manifolds by highlighting the way in which the volume sampling procedure can be viewed as a random walk approximation of the Girsanov change of measure, at least in the Riemannian case (see Appendix~\ref{a:Girsanov}).

For the benefit of exposition, the proofs are collected in Section~\ref{s:proofs}. For the reader's convenience, we collect the results for different structures in Table~\ref{t:table1}; see the appropriate sections for more details and explanation of the notation.

\begin{landscape}

\begin{table}[H]
\centering
\setlength\extrarowheight{1pt}	
\begin{tabular}{|m{23mm}||m{81mm}|m{86mm}|} \hline
\multicolumn{1}{|c||}{Structure} & \multicolumn{1}{c|}{ \textbf{Geodesic RW with volume sampling}} & \multicolumn{1}{c|}{ \textbf{Flow RW with volume sampling}} \\ \hline
 Riemannian &  \begin{tabular}{l} $L_{\omega,c} = \Delta_\omega + (2c-1)\grad(h) \phantom{\sum_{i=1}^n}$ \\[.2cm]  $c=\tfrac{1}{2}$ or $h = \text{const}$ : $L_{\omega,c} = \Delta_\omega$ \\[.2cm] $c=0$ :  $L_{\omega,0} =   \displaystyle \lim_{c \to 0} L_{\omega,c} = \Delta_{\mathcal{R}}$ \\[.2cm]  (see Theorem~\ref{t:limit-Riemannian})\end{tabular} & \begin{tabular}{l} $L_{\omega,c} = \Delta_\omega + c \grad(h) + (c-1)   \sum_{i=1}^{n} \dive_\omega(X_i) X_i$ \\[.2cm] $c=1$ and $h = \text{const}$ : $L_{\omega,1} = \Delta_\omega$ \\[.2cm] $c= 0$ : $L_{\omega,0} =   {\displaystyle \lim_{c \to 0} L_{\omega,c}} = \sum_{i=1}^{n} X_i^2$ \\[.2cm] (see Theorem~\ref{t:limit-Riemannian-fake}) \end{tabular} \\[.2cm] \hline 
 
Heisenberg group $\mathbb{H}_{2d+1}$ & \begin{tabular}{l} $	L_{\omega,c} = \sigma(c) \left(\Delta_{\omega}+ (2c-1) \grad(h) \right) \phantom{\sum_{i=1}^n} \phantom{\sum_{i=1}^n}$ \\[.2cm] $c=\tfrac{1}{2}$ or $h =\text{const}$ : $L_{\omega,c} = \sigma(c) \Delta_\omega$ \\[.2cm] $c \to 0$ : $  \displaystyle\lim_{c \to 0} L_{\omega,c} = 0$ \quad ($\star$) \\[.2cm] (see Theorem~\ref{t:limit-contact-Heis}) \end{tabular} & \begin{center}\begin{tabular}{c} (see below) \end{tabular}\end{center} \\[.2cm] \hline

Contact Carnot group & \begin{tabular}{l} $L_{\omega,c} = \div_{\omega}\circ\grad'+ (2c-1) \grad'(h) \phantom{\sum_{i=1}^n}$ \\[.2cm] $c=\tfrac{1}{2}$ or $h =\text{const}$ : $L_{\omega,c} = \div_{\omega} \circ \grad'$ \\[.2cm] $c \to 0$ : $ \displaystyle \lim_{c \to 0} L_{\omega,c} = 0$ \quad ($\star$) \\[.2cm] (see Theorem~\ref{t:limit-contact-carnot} and Corollary~\ref{cor:intrinsicformula}) \end{tabular} &  \begin{center}\begin{tabular}{c} (see below) \end{tabular}\end{center} \\[.2cm] \hline

General contact & \begin{center}\begin{tabular}{c} Open problem \end{tabular}\end{center} & \begin{tabular}{l} $ L_{\omega,c} = \Delta_\omega + c \grad(h)+ (c-1)  \sum_{i=1}^{k} \dive_\omega(X_i) X_i$ \\[.2cm] $c=1$ and $h = \text{const}$ : $L_{\omega,1} = \Delta_\omega$ \\[.2cm] $c= 0$ : $L_{\omega,0} =   { \displaystyle \lim_{c \to 0} L_{\omega,c}} = \sum_{i=1}^{k} X_i^2$ \\[.2cm] (see Theorem~\ref{t:limit-contact-fake}) \end{tabular} \\[.2cm] \hline
\end{tabular}
\caption{In each cell, $L_{\omega,c}$ is the generator of the limit diffusion associated with the corresponding construction. Here $c \in [0,1]$ is the ratio of the volume sampling, $n= \dim M$ and $k = \rank \distr$. (i) In the Riemannian case $\omega = e^h \mathcal{R}$, where $\mathcal{R}$ is the Riemannian volume. (ii) In the sub-Riemannian case $\omega = e^h \mathcal{P}$, where $\mathcal{P}$ is Popp volume. (iii) Recall that $\Delta_\omega = \dive_\omega \circ \grad$, and is essentially self-adjoint in $L^2(M,\omega)$ if $M$ is complete. (iv) $X_1,\ldots,X_k$ is a local orthonormal frame ($k=n$ in the Riemannian case). (v) For the definition of the constant $\sigma(c)$, see the appropriate theorem. (vi) $\grad'$ is the gradient computed w.r.t.\ to a modified sub-Riemannian structure on the same distribution (see Section~\ref{Sect:IntrinsicFormula}). (vii) The case of $\mathbb{H}_{2d+1}$ is a particular case of contact Carnot groups, where $\grad' = \sigma(c) \grad$. ($\star$) See Section~\ref{s:altconstr} for an alternative construction where one walks only along minimizing geodesics and which, in the limit for $c \to 0$, gives a non-zero operator.} \label{t:table1}
\end{table}

\end{landscape}

\section{Convergence of random walks}\label{s:convergence}

We recall some preliminaries in sub-Riemannian geometry (see \cite{nostrolibro}, but also \cite{montgomerybook,riffordbook,Jea-2014}). 
\begin{definition}
A \emph{(sub-)Riemannian manifold} is a triple $(M,\distr,\metr)$ where $M$ is smooth, connected manifold, $\distr \subset TM$ is a vector distribution of constant rank $k \leq n$ and $\g$ is a smooth scalar product on $\distr$. We assume that $\distr$ satisfies the \emph{H\"ormander's condition}
\begin{equation}
\spn \{ [X_{i_1},[X_{i_2},[\ldots,[X_{i_{m-1}},X_{i_m}]]] \mid m\geq 0, \quad X_{i_\ell} \in \Gamma(\distr)\}_q  = T_q M, \qquad \forall q \in M.
\end{equation}
\end{definition}
By the Chow-Rashevskii theorem, any two points in $M$ can be joined by a Lipschitz continuous curve whose velocity is a.e.\ in $\distr$. We call such curves \emph{horizontal}. Horizontal curves $\gamma : I \to M$ have a well-defined length, given by
\begin{equation}
\ell(\gamma) = \int_I \|\gamma(t)\| dt,
\end{equation}
where $\| \cdot \|$ is the norm induced by $\g$. The \emph{sub-Riemannian distance} between $p,q \in M$ is
\begin{equation}
d(p,q) = \inf\{ \ell(\gamma) \mid \gamma \text{ horizontal curve connecting $q$ with $p$} \}.
\end{equation}
This distance turns $(M,\distr,\metr)$ into a metric space that has the same topology of $M$. A sub-Riemannian manifold is \emph{complete} if $(M,d)$ is complete as a metric space. In the following, unless stated otherwise, we always assume that (sub-)Riemannian structures under consideration are complete.

(Sub-)Riemannian structures include Riemannian ones, when $k=n$. We use the term ``sub-Riemannian'' to denote structures that are not Riemannian, i.e.\ $k < n$.

\begin{definition}\label{Def:PathSpace}
If $M$ is a (sub-)Riemannian manifold, (following the basic construction of Stroock and Varadhan \cite{SAndV}) let $\Omega(M)$ be the space of continuous paths from $[0,\infty)$ to $M$. If $\gamma\in\Omega(M)$ (with $\gamma(t)$ giving the position of the path at time $t$), then the metric on $M$ induces a metric $d_{\Omega(M)}$ on $\Omega(M)$ by
\[
d_{\Omega(M)}\lp \gamma^1,\gamma^2\rp= \sum_{i=1}^{\infty}\frac{1}{2^i}
\frac{\sup_{0\leq t\leq i}d\lp \gamma^1(t),\gamma^2(t)\rp}{1+\sup_{0\leq t\leq i}d\lp \gamma^1(t),\gamma^2(t)\rp} 
\]
making $\Omega(M)$ into a Polish space. We give $\Omega(M)$ its Borel $\sigma$-algebra. We are primarily interested in the weak convergence of probability measures on $\Omega(M)$.
\end{definition}

A choice of probability measure $P$ on $\Omega(M)$ determines a continuous, random process on $M$, and (in this section) we will generally denote the random position of the path at time $t$ by $q_t$. Moreover, we will use the measure $P$ and the process $q_t$ interchangeably.

We are interested in what one might call bounded-step-size, parabolically-scaled families of random walks, which for simplicity in what follows, we will just call a family of random walks. We will index our families by a ``spatial parameter'' $\eps>0$ (this will be clearer below), and we let $\dt=\eps^2/(2k)$ be the corresponding time step ($k$ is the rank of $\distr$).

\begin{definition}\label{Def:Walk}
A family of random walks on a (sub-)Riemannian manifold $M$, indexed by $\eps>0$ and starting from $q\in M$, is a family of probability measures $P^\eps_q$ on $\Omega(M)$ with $P^{\eps}_q(q^{\eps}_0=q)=1$ and having the following property. For every $\eps$, and every $\tilde{q}\in M$, there exists a probability measure $\Pi_{\tilde{q}}^{\eps}$ on continuous paths $\gamma:[0,\dt]\rightarrow M$ with $\gamma(0)=\tilde{q}$ such that for every $m=0,1,2,\ldots$, the distribution of $q^{\eps}_{[m\dt,(m+1)\dt]}$ under $P^{\eps}_q$ is given by $\Pi_{q^{\eps}_{m\dt}}^{\eps}$, independently of the position of the path $q^{\eps}_t$ prior to time $m\dt$. Further,  there exists some constant $\kappa$, independent of $\tilde{q}$ and $\eps$, such that the length of $\gamma_{[0,\dt]}$ is almost surely less than or equal to $\kappa\eps$ under $\Pi_{\tilde{q}}^{\eps}$. (So the position of the path at times $m\dt$ for $m=0,1,2,\ldots$ is a Markov chain, starting from $q$, with transition probabilities $P_q^{\eps}\lp q^{\eps}_{(m+1)\dt}\in A\mid q^{\eps}_{m\dt}=\tilde{q} \rp = \Pi_{\tilde{q}}^{\eps}\lp \gamma_{\dt}\in A\rp$ for any Borel $A\subset M$.
\end{definition}

\begin{rmk}\label{r:remfordefinition}
In what follows $\Pi_{\tilde{q}}^{\eps}$ will, in most cases, be supported on paths of length exactly $\eps$ (allowing us to take $\kappa=1$). For example, on a Riemannian manifold, one might choose a direction at $q^{\eps}_{m\dt}$ at random and then follow a geodesic in this direction for length $\eps$ (and in time $\dt$). Alternatively, on a Riemannian manifold with a global orthonormal frame, one might choose a random linear combination of the vectors in the frame, still having length 1, and then flow along this vector field for length $\eps$. In both of these cases, $\Pi_{\tilde{q}}^{\eps}$ is itself built on a probability measure on the unit sphere in $T_{\tilde{q}}M$ according to a kind of scaling by $\eps$. These walks, and variations and sub-Riemannian versions thereof, form the bulk of what we consider, and should be sufficient to illuminate the definition.

While the introduction of the ``next step'' measure $\Pi^{\eps}_{\tilde{q}}$ is suitable for the general definition and accompanying convergence result, it is overkill for the geometrically natural steps that we consider. Instead, we will describe the steps of our random walks in simpler geometric terms (as in the case of choosing a random geodesic segment of length $\eps$ just mentioned), and leave the specification of $\Pi^{\eps}_{\tilde{q}}$ implicit, though in a straightforward way.
\end{rmk}

\begin{rmk}
All of the random walks we consider will be horizontal, in the sense that $\Pi_{\tilde{q}}^{\eps}$ is supported on horizontal curves. (In the Riemannian case, this, of course, is vacuous.) So while the diffusions we will get below as limits of such random walks will not be horizontal insofar as they are supported on paths that are not smooth enough to satisfy the definition of horizontal given above, they nonetheless are limits of horizontal processes.
\end{rmk}

We note that, for some constructions like that of solutions to a Stratonovich SDE, there need not be a metric on $M$, but instead a smooth structure is sufficient. Unfortunately, the machinery of convergence of random walks in Theorem \ref{t:convergence} below is formulated in terms of metrics, and thus we will generally proceed by choosing some (Riemannian or sub-Riemannian) metric on $M$ when desired. However, note that if $M$ is compact, any two Riemannian metrics induce Lipschitz-equivalent distances on $M$, and thus the induced distances on $\Omega(M)$ are comparable. This means that the resulting topologies on $\Omega(M)$ are the same, and thus statements about the convergence of probability measures on $\Omega(M)$ (which is how we formalize the convergence of random walks) do not depend on what metric on $M$ is chosen. This suggests that a more general framework could be developed, avoiding the need to introduce a metric on $M$ when the smooth structure should suffice, but such an approach will not be pursued here.

\begin{definition}\label{d:operator}
Let $\eps>0$. To the family of random walks $q_t^\eps$ (in the sense of Definition \ref{Def:Walk}, and with the above notation), we associate the family of smooth operators on $C^\infty(M)$
\begin{equation}
(L^\eps\phi)(q) := \frac{1}{\dt}\mathbb{E}[\phi(q_{\dt}^\eps) - \phi(q)\mid q_0^\eps =q], \qquad \forall q \in M.
\end{equation}
\end{definition}

\begin{definition}
Let $L$ be a differential operator on $M$. We say that a family $L^\varepsilon$ of differential operators converge to $L$ if for any $\phi \in C^\infty(M)$ we have $L^\varepsilon \phi \to L \phi$ uniformly on compact sets. In this case, we write $L^\varepsilon \to L$.
\end{definition}

Let $L$ be a smooth second-order differential operator with no zeroth-order term. If the principal symbol of $L$ is also non-negative definite, then there is a unique diffusion associated to $L$ starting from any point $q\in M$, at least up until a possible explosion time. However, since our analysis in fundamentally local, we will assume that the diffusion does not explode. In that case, this diffusion is given by the measure $P_q$ on $\Omega(M)$ that solves the martingale problem for $L$, so that
\[
\phi(q_t) - \int_0^t L\phi(q_s)\, ds
\]
is a martingale under $P_q$ for any smooth, compactly supported $\phi$, and $P_q\lp q_0=q\rp=1$.

\begin{theorem}\label{t:convergence}
Let $M$ be a (sub-)Riemannian manifold, let $P^\eps_q$ be the probability measures on $\Omega(M)$ corresponding to a sequence of random walks $q_t^\eps$ (in the sense of Definition \ref{Def:Walk}), with $q^\eps_0  =q$, and let $L^\eps$ be the associated family of operators. Suppose that $L^\eps \to L^0$ (in the sense of Definition \ref{d:operator}), where $L^0$ is a smooth second-order operator with non-negative definite principal symbol and without zeroth-order term. Further, suppose that the diffusion generated by $L$, which we call $q_t^0$, does not explode, and let $P_q^0$ be the corresponding probability measure on $\Omega(M)$ starting from $q$. Then $P^\eps_q \to P_q^0$ as $\eps \to 0$, in the sense of weak convergence of probability measures (see Definition~\ref{Def:PathSpace}).
\end{theorem}

\begin{proof}
The theorem is a special case of \cite[Thm. 70, Rmk. 26]{OurUrPaper}. First note that a random walk $q_t^\eps$ as described here corresponds to a random walk $X_{t}^h$ in the notation of \cite{OurUrPaper}, with $h = \eps^2/2k$, and with each step being given either by a continuous curve (which may or may not be a geodesic), as addressed in Remark 26. Every random walk in our class has the property that, during any step, the path never goes more than distance $\kappa\eps$ from the starting point of the step for some fixed $\kappa>0$, by construction, and this immediately shows that every random walk in our class satisfies Eq.~(19) of \cite{OurUrPaper}. 
Then all the assumptions of \cite[Thm. 70]{OurUrPaper} are satisfied and the conclusion follows, namely $P^\eps_q \to P_q^0$ as $\eps \to 0$.
\end{proof}


\section{Geodesic random walks in the Riemannian setting}\label{s:riem-RW}

\subsection{Ito SDEs via geodesic random walks}\label{s:ito-intro}

Let $(M,\g)$ be a Riemannian manifold. We consider a set of smooth vector fields, and since we are interested in local phenomena, we assume that the $V_i$ have bounded lengths and that $(M,\g)$ is complete. We now consider the Ito SDE \begin{equation}\label{Eqn:ItoSDE}
dq_t = \sum_{i=1}^{k} V_i\lp q_t\rp d(\sqrt{2} w_t^i) , \qquad q_0=q,
\end{equation}
for some $q\in M$, where $w_t^1,\ldots,w_t^k$ are independent, one-dimensional Brownian motions\footnote{One approach to interpreting and solving \eqref{Eqn:ItoSDE}, as well as verifying that $q_t$ will be a martingale, is via lifting it to the bundle of orthonormal frames; see the first two chapters of \cite{Hsu} for background on stochastic differential geometry, connection-martingales, and the bundle of orthonormal frames. Alternatively, \cite[Chapter 7]{emery} gives a treatment of Ito integration on manifolds.}. To construct a corresponding sequence of random walks, we choose a random vector $V=\beta_1V_1+\beta_2V_2+\cdots+\beta_kV_k$ by choosing $(\beta_1,\ldots,\beta_k)$ uniformly from the unit sphere. Then, we follow the geodesic $\gamma(s)$ determined by $\gamma(0)=q$ and $\gamma^{\prime}(0)=\frac{2k}{\eps}V$ for time $\dt=\eps^2/(2k)$. Equivalently, we travel a distance of $\eps |V|$ in the direction of $V$ (along a geodesic). This determines the first step, $q^{\eps}_t$ with $t\in[0,\dt]$, of a random walk (and thus, implicitly, the measure $\Pi^{\eps}_q$). Determining each additional step in the same way produces a family of piecewise geodesic random walks $q^{\eps}_t$, $t\in[0,\infty)$, which we call the \emph{geodesic random walk} at scale $\eps$ associated with the SDE \eqref{Eqn:ItoSDE} (in terms of Definition \ref{Def:Walk}, $\kappa=\sup_{q,(\beta_1,\ldots,\beta_k)} V$).

We now study the convergence of this family of walks as $\eps\rightarrow 0$. Let $x_1,\ldots,x_n$ be Riemannian normal coordinates around $q_0^\eps=q$, and write the random vector $V$ as
\begin{equation}\label{Eqn:VExpand}
V(x)= \sum_{m=1}^k \beta_m V_m(x) = \sum_{m=1}^k \beta_m \sum_{i=1}^n V_m^i \partial_{i}
+O(r) = \sum_{i=1}^n A^i \partial_{ i} + O(r),
\end{equation}
where $r = \sqrt{x_1^2+\ldots+x_n^2}$. In normal coordinates, Riemannian geodesics correspond to Euclidean geodesics, and thus $\gamma_V(t)$ has $i$-th coordinate $A^i t$. In particular, for any smooth function $\phi$ we have
\begin{equation}
\phi(\gamma_V(\eps)) - \phi(q) = \sum_{i=1}^n A^i (\partial_{i}\phi)(q) \eps + \frac{1}{2}\sum_{i,j=1}^n A^i A^j (\partial_i \partial_j \phi)(q) \eps^2 + O(\eps^3).
\end{equation}
Averaging w.r.t.\ the uniform probability measure on the sphere $\sum_{i=1}^k \beta_i^2 =1$, we obtain
\begin{equation}\label{Eqn:ItoConv}
\begin{split}
L^{\eps} :=  \frac{1}{\dt} \E\lb \phi\lp q^{\eps}_{\dt}\rp -\phi(q)\left| q^{\eps}_0=q \right.\rb & \rightarrow \sum_{m=1}^k  \sum_{i,j=1}^n V^i_m V^j_m (\partial_{i}\partial_{j}\phi)(q) \\
&= \sum_{m=1}^k \nabla^2_{V_m,V_m}(q) \qquad \text{as } \eps \to 0,
\end{split}\end{equation}
where $\nabla^2$ denotes the Hessian, with respect to Levi-Civita connection, and where we recall that $\sum_{j=1}^n V_m^j \partial_{j} = V_m(q)$ and the $x_i$ are a system of normal coordinates at $q$. The right-hand side of \eqref{Eqn:ItoConv} determines a second-order operator which is independent of the choice of normal coordinates (and thus depends only on the $V_i$). Moreover, this same construction works at any point, and thus we have a second-order operator $L = \lim_{\eps \to 0} L^\eps$ on all of $M$. Because the $V_i$ are smooth, so is $L$ (and the convergence is uniform on compacts).

We see that the martingale problem associated to $L$ has a unique solution (at least until explosion, but since we are interested in local questions, we can assume that there is no explosion). Further, this solution is the law of the process $q^0_t$ that solves \eqref{Eqn:ItoSDE}. If we again let $P^{\eps}$ and $P^0$ be the probability measures on $\Omega(M)$ corresponding to $q^{\eps}_t$ and $q^0_t$, respectively, Theorem \ref{t:convergence} implies that $P^{\eps}\rightarrow P^0$ (weakly) as $\eps\rightarrow 0$.

Of course, we see that our geodesic random walks, as well as the diffusion $q^0$ and thus the interpretation of the SDE \eqref{Eqn:ItoSDE}, depend on the Riemannian structure. This is closely related to the fact that neither Ito SDEs, normal coordinates, covariant derivatives, nor geodesics are preserved under diffeomorphisms, in general, and to the non-standard calculus of Ito's rule for Ito integrals, in contrast to Stratonovich integrals. Note also that in this construction, it would also be possible to allow $k>n$.

The most important special case of a geodesic random walk is when $k =n$ and the vector fields $V_1,\ldots,V_n$ are an orthonormal frame. In that case, $q^{\eps}_t$ is an isotropic random walk, as described in \cite{OurUrPaper} (see also \cite{PinksyRiem} for a related family of processes) and
\begin{equation}\label{eq:I}
L^\eps \to \Delta,
\end{equation}
where $\Delta = \dive \circ \grad$ is the Laplace-Beltrami operator (here the divergence is computed with respect to the Riemannian volume). In particular $q^0_t$ is  Brownian motion on $M$, up to time-change by a factor of 2.

If we further specialize to Euclidean space, we see that the convergence of the random walk to Eucldiean Brownian motion is just a special case of Donsker's invariance principle. The development of Brownian motion on a Riemannian manifold via approximations is also not new; one approach can be found in \cite{StroockGeo}.

\subsection{Volume sampling through the exponential map}\label{s:volumesampling-intro}

Let $(M,\metr)$ be a $n$-dimensional Riemannian manifold equipped with a general volume form $\omega$, that might be different from the Riemannian one $\mathcal{R}$. This freedom is motivated by the forthcoming applications to sub-Riemannian geometry, where there are several choices of intrinsic volumes and in principle there is not a preferred one \cite{ABB-Hausdorff,nostropopp}. Besides, even in the Riemannian case, one might desire to study operators which are symmetric w.r.t.\ a general measure $\omega = e^h \mathcal{R}$.

We recall that the gradient $\grad(\phi)$ of a smooth function depends only on the Riemannian structure, while the divergence $\dive_\omega(X)$ of a smooth vector field depends on the choice of the volume. In this setting we introduce an intrinsic diffusion operator, symmetric in $L^2(M,\omega)$, with domain $C^\infty_c(M)$ as the divergence of the gradient:
\begin{equation}
\Delta_\omega:=\div_\omega\circ\grad  = \sum_{i=1}^n X_i^2 + \dive_\omega(X_i)X_i,
\end{equation}
where in the last equality, which holds locally, $X_1,\ldots,X_n$ is a orthonormal frame. Recall that if $\omega$ and $\omega'$ are proportional, then $\Delta_\omega=\Delta_{\omega'}$.

If one would like to define a random walk converging to the diffusion associated to $\Delta_\omega$, one should make the construction in such a way that the choice of the volume enters in the definition of the walk. One way to do this is to ``sample the volume along the walk''. For all $s\geq 0$, consider the Riemannian exponential map
\begin{equation}
\exp_q(s;\cdot): S_q M \to M, \qquad q \in M,
\end{equation}
where $S_q M \subset T_q M$ is the Riemannian tangent sphere. In particular, for $v \in S_q M$, $\gamma_v(s) = \exp_q(s;v)$ is the unit-speed geodesic starting at $q$ with initial velocity $v$. Then $|\iota_{\dot\gamma_v(s)}\omega|$ is a density\footnote{If $\eta$ is an $m$-form on an $m$-dimensional manifold, the symbol $|\eta|$ denotes the associated density, in the sense of tensors.} on the Riemannian sphere of radius $s$. By pulling this back through the exponential map, we obtain a probability measure on $S_qM$ that ``gives more weight to geodesics arriving where there is more volume''.
\begin{definition}\label{d:Riemmes}
For any $q \in M$, and $\eps >0$, we define the family of densities $\mu_q^\eps$ on $S_qM$
\begin{equation}
\mu^{\eps}_q(v) := \frac{1}{N(q,\eps)}\left\lvert(\exp_q(\eps;\cdot)^* \iota_{\dot\gamma_v(\eps)}\omega)(v)\right\rvert, \qquad \forall v \in S_qM,
\end{equation}
where $N(q,\eps)$ is such that $\int_{S_qM} \mu_{q}^\eps = 1$. For $\eps =0$, we set $\mu^0_q$ to be the standard normalized Riemannian density form on $S_qM$.
\end{definition}
\begin{rmk}\label{r:absolutevalue}
For any fixed $q \in M$, and for sufficiently small $\eps>0$, the Jacobian determinant of $\exp_q(\eps;\cdot)$ does not change sign, hence the absolute value in definition~\ref{d:Riemmes} is not strictly necessary to obtain a well defined probability measure on $S_qM$. By assuming that the sectional curvature $\mathrm{Sec} \leq K$ is bounded from above, one can globally get rid of the need for the absolute value, as conjugate lengths are uniformly separated from zero.
\end{rmk}
Now we define a random walk $b^\eps_{t}$ as follows:
\begin{equation}\label{eq:backwards-Ito}
b_{(i+1) \dt}^\eps := \exp_{b_{i \dt}^\eps}(\eps; v), \qquad v \in S_{q}M  \text{ chosen with probability }  \mu_{q}^\eps.
\end{equation}
(see Definition~\ref{Def:Walk} and Remark~\ref{r:remfordefinition}). Let $P^\eps_{\omega}$ (we drop the $q$ from the notation as the starting point is fixed) be the probability measure on the space of continuous paths on $M$ starting at $q$ associated with $b_t^\eps$, and consider the associated family of operators\begin{align}
(L_{\omega}^\eps\phi)(q) & :=  \frac{1}{\dt} \mathbb{E}[\phi(b_{\dt}^\eps) - \phi(q) \mid b^\eps_0 =q] \\
& := \frac{1}{\dt}\int_{S_qM} [\phi(\exp_q(\eps;v))-\phi(q)] \mu_q^{\eps}(v), \qquad \forall q \in M,
\end{align}
(see Definition~\ref{d:operator}), for any $\phi \in C^\infty(M)$. A special case of Theorem~\ref{t:limit-Riemannian} gives
\begin{align}\label{eq:BI}
\lim_{\eps  \to 0} L_\omega^\eps = \underbrace{\Delta_{\mathcal{R}} + \grad(h)}_{\Delta_\omega} + \grad (h),
\end{align}
where $\grad(h)$ is understood as a derivation. By Theorem~\ref{t:convergence}, $P^\eps_\omega$ converges to a well-defined diffusion generated by the r.h.s.\ of \eqref{eq:BI}. This result is not satisfactory, as one would prefer $L_\omega^\eps \to \Delta_\omega$. Indeed, in \eqref{eq:BI}, we observe that the correction $2\grad(h)$ provided by the volume sampling construction is twice the desired one (except when $\omega$ is proportional to $\mathcal{R}$).

To address this problem we introduce a parameter $c \in [0,1]$ and consider, instead, the family $\mu^{c\eps}_q$. This corresponds to sampling the volume not at the final point of the geodesic segment, but at an intermediate point. We define a random walk as follows:
\begin{equation}
b_{(i+1)\dt ,c}^\eps := \exp_{q_{i \dt ,c}^\eps}(\eps, v), \qquad v \in S_qM  \text{ with probability }  \mu_q^{c\eps},
\end{equation}
that we call the \emph{geodesic random walk with volume sampling} (with volume $\omega$ and sampling ratio $c$).

\begin{rmk}
The case $c = 0$ does not depend on the choice of $\omega$ and reduces to the construction of Section~\ref{s:ito-intro}. The case $c=1$ corresponds to the process of Equation \eqref{eq:backwards-Ito}. 
 \end{rmk}
For $\eps>0$, let $P^\eps_{\omega,c}$ be the probability measure on the space of continuous paths on $M$ associated with the process $b_{t,c}^\eps$, and consider the family of operators
\begin{equation}
\begin{aligned}
(L_{\omega,c}^\eps\phi)(q) & :=  \frac{1}{\dt} \mathbb{E}[\phi(b_{\dt}^\eps) - \phi(q) \mid b^\eps_0 =q] \\
& := \frac{1}{\dt}\int_{S_qM} [\phi(\exp_q(\eps;v))-\phi(q)] \mu_q^{c\eps}(v), \qquad \forall q \in M,
\end{aligned}
\end{equation}
for any $\phi \in C^\infty(M)$. The family of Riemannian geodesic random walks with volume sampling converges to a well-defined diffusion, as follows.
\begin{theorem}\label{t:limit-Riemannian}
Let $(M,\g)$ be a complete Riemannian manifold with volume $\omega = e^h\mathcal{R}$, where $\mathcal{R}$ is the Riemannian one, and $h \in C^\infty(M)$. Let $c \in [0,1]$. Then $L_{\omega,c}^\eps \to L_{\omega,c}$, where
\begin{equation}\label{eq:ito-c}
L_{\omega,c} = \Delta_\omega +  (2c-1)\grad(h).
\end{equation}
Moreover $P^\eps_{\omega,c} \to P_{\omega,c}$ weakly, where $P_{\omega,c}$ is the law of the diffusion associated with $L_{\omega,c}$ (which we assume does not explode).
\end{theorem}
\begin{rmk}
We have these alternative forms of \eqref{eq:ito-c}, obtained by unraveling the definitions:
\begin{align}
L_{\omega,c} &= \Delta_{e^{(2c-1)h }\omega}  = \Delta_{e^{2c h}\mathcal{R}} =\Delta_{\mathcal{R}} +  2c\grad(h) \\
& = \sum_{i=1}^n X_i^2+ \left( 2c \dive_\omega(X_i) +(1- 2c)\dive_{\mathcal{R}}(X_i)\right)X_i ,
\end{align}
where, in the last line, $X_1,\ldots,X_n$ is a local orthonormal frame.
\end{rmk}

\begin{figure}\label{Fig:VolSampling}
\includegraphics[scale=1.4]{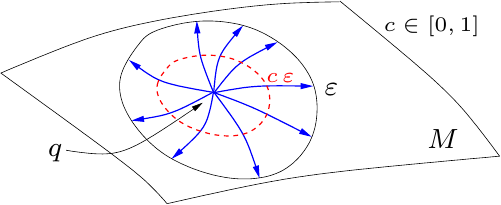}
\caption{Geodesic random walk with sampling of the volume $\omega$ and ratio $c$. For each $\eps$, the paths of the walk are piecewise-smooth geodesics.}
\end{figure}

As a simple consequence of \eqref{eq:ito-c} or its alternative formulations, we have the following statement, which appears to be new even in the Riemannian case.
\begin{cor}\label{c:self-adj-Riem}
Let $(M,\g)$ be a complete Riemannian manifold. The operator $L_{\omega,c}$ with domain $C^\infty_c(M)$ is essentially self-adjoint in $L^2(M,\omega)$ if and only at least one of the following two conditions hold:
\begin{itemize}
\item[(i)] $c=1/2$;
\item[(ii)] $\omega$ is proportional to the Riemannian volume (i.e.\ $h$ is constant).
\end{itemize}
\end{cor}
The previous discussion stresses the particular role played by the Riemannian volume. Not only does it coincide with the Hausdorff measure, but according to the above construction, it is the only volume (up to  constant rescaling) that gives the correct self-adjoint operator for \emph{any} choice of the parameter $c$.

\begin{rmk}
If we want the volume-sampling scheme to produce the Laplacian w.r.t. the volume $\omega$ being sampled, we should take $c=1/2$. With hindsight, this might not be surprising. By linearity, we see that sampling with $c=1/2$ is equivalent to sampling the volume along the entire step, uniformly w.r.t time (recall that the geodesics are traversed with constant speed), rather than privileging any particular point along the path.
\end{rmk}
\begin{rmk}
One can prove that the limiting operator corresponding to the geodesic random walk with volume sampling ratio $c=1$ is equal, up to a constant (given by the ratio of the area of the Euclidean unit sphere and the volume of the unit ball in dimension $n$), to the limiting operator corresponding to a more general class of random walk where we step to points of the metric ball $B_{q}(\varepsilon)$ of radius $\varepsilon$, uniformly w.r.t. its normalized volume $\omega/\omega(B_{q}(\varepsilon))$. This kind of random walk for the Riemannian volume measure has also been considered in \cite{LebeauRiem}, in relation with the study of its spectral properties.
\end{rmk}

\section{Geodesic random walks in the sub-Riemannian setting}\label{s:SR-RW}

We want to define a sub-Riemannian version of the geodesic random walk with volume sampling, extending the Riemannian construction of the previous section. Recall the definition of (sub-)Riemannian manifold in Section~\ref{s:convergence}.

\subsection{Geodesics and exponential map} As in Riemannian geometry, \emph{geode\-sics} are horizontal curves that have constant speed and locally minimize the length between their endpoints. Define the \emph{sub-Riemannian Hamiltonian} $H: T^*M \to \R$ as
\begin{equation}
H(\lambda) := \frac{1}{2} \sum_{i=1}^k \langle\lambda,X_i\rangle^2,
\end{equation}
for any local orthonormal frame $X_1,\ldots,X_k \in \Gamma(\distr)$. Let $\sigma$ be the natural symplectic structure on $T^*M$, and $\pi: T^*M \to M$. The \emph{Hamiltonian vector field} $\vec{H}$ is the unique vector field on $T^*M$ such that $dH = \sigma(\cdot,\vec{H})$. Then the Hamilton equations are 
\begin{equation}\label{eq:hamilton}
\dot{\lambda}(t) = \vec{H}(\lambda(t)).
\end{equation}
Solutions of \eqref{eq:hamilton} are smooth curves on $T^*M$, and their projections $\gamma(t):=\pi(\lambda(t))$ on $M$ will be geodesics. In the Riemannian setting, all geodesics can be recovered uniquely in this way. In the sub-Riemannian one, this is no longer true, as \emph{abnormal geodesics} can appear, which are geodesics that might not come from projections of solutions to \eqref{eq:hamilton}.

For any $\lambda \in T^*M$ we consider the geodesic $\gamma_\lambda(t)$, obtained as the projection of the solution of \eqref{eq:hamilton} with initial condition $\lambda(0) = \lambda$. Observe that the Hamiltonian function, which is constant on $\lambda(t)$, measures the speed of the associated geodesic: 
\begin{equation}
2H(\lambda) = \|\dot\gamma_\lambda(t)\|^2,\qquad \lambda \in T^*M.
\end{equation}
Since $H$ is fiber-wise homogeneous of degree $2$, we have the following rescaling property:
\begin{equation}
\gamma_{\alpha \lambda}(t) = \gamma_{\lambda}(\alpha t),\qquad \alpha > 0.
\end{equation}
This justifies the restriction to the subset of initial covectors lying in the level set $2H=1$. 
\begin{definition}
The \emph{unit cotangent bundle} is the set of initial covectors such that the associated geodesic has unit speed, namely
\begin{equation}
\cyl := \{\lambda \in T^*M \mid 2H(\lambda) = 1\} \subset T^*M.
\end{equation}
\end{definition}
For any $\lambda \in \cyl$, the geodesic $\gamma_\lambda(t)$ is \emph{parametrized by arc-length}, namely $\ell(\gamma|_{[0,T]}) = T$.
\begin{rmk}
We stress that, in the genuinely sub-Riemannian case, $H|_{T_q^*M}$ is a degenerate quadratic form. It follows that the fibers $\cyl_q$ are non-compact cylinders, in sharp contrast with the Riemannian case (where the fibers $\cyl_q$ are spheres).
\end{rmk}
For any $\lambda \in \cyl$, the \emph{cut time} $t_c(\lambda)$ is defined as the time at which $\gamma_\lambda(t)$ loses optimality
\begin{equation}
t_c(\lambda) := \sup \{t>0 \mid d(\gamma_\lambda(0),\gamma_\lambda(t)) = t\}.
\end{equation}
In particular, for a fixed $\eps>0$ we define
\begin{equation}\label{eq:cotingj}
\cyl_q^\eps := \{\lambda \in \cyl_q \mid t_c(\lambda) \geq \eps \} \subset \cyl_q,
\end{equation}
as the set of unit covector such that the associated geodesic is optimal up to time $\eps$.
\begin{definition}
Let $D_q \subseteq [0,\infty) \times \cyl_q$ the set of the pairs $(t,\lambda)$ such that $\gamma_\lambda$ is well defined up to time $t$. The \emph{exponential map} at $q \in M$ is the map  $\exp_q: D_q \to M$ that associates with $(t,\lambda)$ the point $\gamma_\lambda(t)$.
\end{definition}
Under the assumption that $(M,d)$ is complete, by the (sub-)Riemannian Hopf-Rinow Theorem (see, for instance, \cite{nostrolibro,riffordbook}), we have that any closed metric ball is compact and normal geodesics can be extended for all times, that is $D_q = [0,\infty) \times \cyl_q$, for all $q \in M$.

\subsection{Sub-Laplacians}
For any function $\phi \in C^\infty(M)$, the \emph{horizontal gradient} $\grad(\phi) \in \Gamma(\distr)$ is, at each point, the horizontal direction of steepest slope of $\phi$, that is
\begin{equation}\label{eq:grad}
\g(\grad(\phi),X) = \langle d\phi, X\rangle, \qquad \forall X \in \Gamma(\distr).
\end{equation}
Since in the Riemannian case this coincides with the usual gradient, this notation will cause no confusion. If $X_1,\ldots,X_k$ is a local orthonormal frame, we have
\begin{equation}
\grad(\phi) = \sum_{i=1}^k X_i(\phi) X_i.
\end{equation}
For any fixed volume form $\omega \in \Lambda^n M$ (or density if $M$ is not orientable), the \emph{divergence} of a smooth vector field $X$ is defined by the relation $\mathcal{L}_X \omega = \dive_{\omega}(X)$, where $\mathcal{L}$ denotes the Lie derivative. Notice that the sub-Riemannian structure does not play any role in the definition of $\dive_\omega$. Following \cite{montgomerybook,laplacian}, the \emph{sub-Laplacian} on $(M,\distr,\g)$ associated with $\omega$ is 
\begin{equation}\label{eq:sublap}
\Delta_\omega := \div_{\omega}\circ \grad  = \sum_{i=1}^k X_i^2 + \dive_\omega(X_i)X_i,
\end{equation}
where in the last equality, which holds locally, $X_1,\ldots,X_k$ is a orthonormal frame. Again, if $\omega$ and $\omega'$ are proportional, then $\Delta_\omega=\Delta_{\omega'}$.

The sub-Laplacian is symmetric on the space $C^\infty_c(M)$ of smooth functions with compact support with respect to the $L^2(M,\omega)$ product. 
If $(M,d)$ is complete and there are no non-trivial abnormal minimizers, then $\Delta_\omega$ is essentially self-adjoint on $C^\infty_c(M)$ and has a smooth positive heat kernel \cite{strichartz,strichartzerrata}.

The sub-Laplacian will be intrinsic if we choose an intrinsic volume. See \cite[Sec. 3]{OurUrPaper} for a discussion of intrinsic volumes in sub-Riemannian geometry. A natural choice, at least in the equiregular setting, is Popp volume \cite{nostropopp,montgomerybook}, which is smooth. Other choices are possible, for example the Hausdorff or the spherical Hausdorff volume which, however, are not always smooth \cite{ABB-Hausdorff}. For the moment we let $\omega$ be a general smooth volume.

\subsection{The sub-Riemannian geodesic random walk with volume sampling}

In contrast with the Riemannian case, where $S_qM$ has a well defined probability measure induced by the Riemannian structure, we have no such construction on $\cyl_q$. Thus, it is not clear how to define a geodesic random walk in the sub-Riemannian setting. 

For $\eps>0$, consider the sub-Riemannian exponential map
\begin{equation}
\exp_q(\eps;\cdot): \cyl_q \to M, \qquad q \in M.
\end{equation}
If $\lambda \in \cyl_q$, then $\gamma_\lambda(\eps) = \exp_q(\eps;\lambda)$ is the associated unit speed geodesic starting at $q$.

One wishes to repeat Definition~\ref{d:Riemmes}, using the exponential map to induce a density on $\cyl_q$, through the formula $\mu_q^\eps(\lambda) \propto |(\exp_q(\eps;\cdot)^* \iota_{\dot\gamma_\lambda(\eps)}\omega)(\lambda)|$. 
However, there are non-trivial difficulties arising in the genuine sub-Riemannian setting.
\begin{itemize}
\item The exponential map is not a local diffeomorphism at $\eps =0$, and Riemannian normal coordinates are not available. This tool is used for proving the convergence of walks in the Riemannian setting;
\item Due to the presence of zeroes in the Jacobian determinant of $\exp_q(\eps;\cdot)$ for arbitrarily small $\eps$, the absolute value in the definition of $\mu_q^\eps$ is strictly necessary (in contrast with the Riemannian case, see Remark~\ref{r:absolutevalue}).
\item Since $\cyl_q$ is not compact, there is no guarantee that $\int_{\cylsmall_q} \mu_q^\eps < +\infty$;
\end{itemize}
Assuming that $\int_{\cylsmall_q} \mu_q^\eps < + \infty$, we generalize Definition~\ref{d:Riemmes} as follows.
\begin{definition}\label{d:mesIto}
For any $q \in M$, and $\eps >0$, we define the family of densities $\mu_q^\eps$ on $\cyl_q$
\begin{equation}
\mu^{\eps}_q(\lambda) := \frac{1}{N(q,\eps)}\left\lvert(\exp_q(\eps;\cdot)^* \iota_{\dot\gamma_\lambda(\eps)}\omega)(\lambda)\right\rvert, \qquad \forall \lambda \in \cyl_q,
\end{equation}
where $N(q,\eps)$ is fixed by the condition $\int_{\cylsmall_q} \mu_{q}^{\eps} = 1$.
\end{definition}

As we did in Section~\ref{s:volumesampling-intro}, and for $c \in (0,1]$, we build a random walk 
\begin{equation}
b_{(i+1)\dt,c}^{\eps}:= \exp_{b_{i \dt,c}^{\eps}}(\eps;\lambda), \qquad \lambda \in \cyl_q \text{ chosen with probability } \mu_q^{c\eps}.
\end{equation}

Let $P^\eps_{\omega,c}$ be the associated probability measure on the space of continuous paths on $M$ starting from $q$, and consider the corresponding family of operators, which in this case is
\begin{equation}\label{eq:operatoreps}
\begin{aligned}
(L_{c,\omega}^\eps\phi)(q) & =\frac{1}{\dt}\mathbb{E}[\phi(b_{\dt,c}^\eps)-\phi(q)\mid b_{0,c}^\eps = q]\\
&= \frac{1}{\dt}\int_{\cyl_q} [\phi(\exp_q(\eps;\lambda))-\phi(q)] \mu_q^{c\eps}(\lambda), \qquad \forall q \in M,
\end{aligned}
\end{equation}
for any $\phi \in C^\infty(M)$. Clearly when $k=n$, \eqref{eq:operatoreps} is the same family of operators associated with a Riemannian geodesic random walk with volume sampling discussed in Section~\ref{s:volumesampling-intro}, and this is why - without risk of confusion - we used the same symbol.
\begin{rmk}
As mentioned, in sub-Riemannian geometry abnormal geodesics may appear. More precisely, one may have \emph{strictly abnormal geodesics} that do not arise as projections of solutions of \eqref{eq:hamilton}. The class of random walks that we have defined never walk along these trajectories, but can walk along abnormal segments that are not strictly abnormal.

The (minimizing) Sard conjecture states that the set of endpoints of strictly abnormal (minimizing) geodesics starting from a given point has measure zero in $M$. However, this remains a hard open problem in sub-Riemannian geometry \cite{AAA-openproblems}. See also \cite{Sard-prop,Sard-Rif-Trel,agrasmooth} for recent progress on the subject.
\end{rmk}

Checking the convergence of \eqref{eq:operatoreps} is difficult in the general sub-Riemannian setting ($k< n$), in part due to the difficulties outlined above. We treat in detail the case of contact Carnot groups, where we find some surprising results. These structures are particularly important as they arise as Gromov-Hausdorff tangent cones of contact sub-Riemannian structures \cite{bellaiche,mitchell}, and play the same role of Euclidean space in Riemannian geometry.

\subsection{Contact Carnot groups}\label{s:contactcarnot}

Let $M = \R^{2d+1}$, with coordinates $(x,z) \in \R^{2d}\times \R$. Consider the following global vector fields
\begin{equation}
X_i = \partial_{x_i} - \frac{1}{2} (A x)_i \partial_z, \qquad i=1,\ldots,2d,
\end{equation}
where
\begin{equation}
A = \begin{pmatrix} \alpha_1 J & & \\
& \ddots & \\
& & \alpha_d J
\end{pmatrix}, \qquad J = \begin{pmatrix}
0 & -1 \\
1 & 0
\end{pmatrix},
\end{equation}
is a skew-symmetric, non-degenerate matrix with singular values $0 < \alpha_1 \leq \ldots \leq \alpha_d$. A \emph{contact Carnot group} is the sub-Rieman\-nian structure on $M = \R^{2d+1}$ such that $\distr_q = \spn\{X_1,\ldots,X_{2d}\}_q$ for all $q \in M$, and $\g(X_i,X_j) = \delta_{ij}$. Notice that
\begin{equation}
[X_i,X_j]= A_{ij} \partial_z.
\end{equation}
Set $\mathfrak{g}_1 := \spn\{X_1,\ldots,X_{2d}\}$ and $\mathfrak{g}_2 := \spn \{\partial_z\}$. The algebra $\mathfrak{g}$ generated by the $X_i$'s and $\partial_z$ admits a nilpotent stratification of step $2$, that is
\begin{equation}
\mathfrak{g} = \mathfrak{g}_1 \oplus \mathfrak{g}_2, \qquad \mathfrak{g}_1,\mathfrak{g}_2 \neq \{0\},
\end{equation}
with
\begin{equation}
[\mathfrak{g}_1,\mathfrak{g}_1] = \mathfrak{g}_{2}, \qquad \text{and} \qquad [\mathfrak{g}_1,\mathfrak{g}_2] = [\mathfrak{g}_2,\mathfrak{g}_2]= \{0\}.
\end{equation}
There is a unique connected, simply connected Lie group $G$ such that $\mathfrak{g}$ is its Lie algebra of left-invariant vector fields. The group exponential map,
\begin{equation}
\mathrm{exp}_{G} : \mathfrak{g} \to G,
\end{equation}
associates with $v \in \mathfrak{g}$ the element $\gamma(1)$, where $\gamma: [0,1] \to G$ is the unique integral curve of the vector field $v$ such that $\gamma(0) = 0$. Since $G$ is simply connected and $\mathfrak{g}$ is nilpotent, $\mathrm{exp}_G$ is a smooth diffeomorphism. Thus we can identify $G \simeq \R^{2d+1}$ equipped with a polynomial product law $\star$ given by
\begin{equation}
(x,z) \star (x',z') = \left(x+x',z+z' + \tfrac{1}{2} x^* A x'\right).
\end{equation}
Denote by $L_q$ the left-translation $L_q (p) := q \star p$. The fields $X_i$ are left-invariant, and as a consequence the sub-Riemannian distance is left-invariant as well, in the sense that $d(L_q(p_1),L_q(p_2)) = d(p_1,p_2)$.

\begin{rmk}
As consequence of left-invariance, contact Carnot groups are complete as metric spaces. Moreover all abnormal minimizers are trivial. Hence, for each volume $\omega$, the operator $\Delta_\omega$ with domain $C_c^\infty(M)$ is essentially self-adjoint in $L^2(M,\omega)$.
\end{rmk}

\begin{example}\label{ex:heis}
The $2d+1$ dimensional \emph{Heisenberg group} $\mathbb{H}_{2d+1}$, for $d \geq 1$, is the contact Carnot group with $\alpha_1=\ldots=\alpha_d = 1$.
\end{example}

\begin{example}\label{ex:biheis}
The \emph{bi-Heisenberg group} is the $5$-dimensional contact Carnot group with $0<\alpha_1 < \alpha_2$. That is, $A$ has two distinct singular values.
\end{example}

A natural volume is the Popp volume $\mathcal{P}$. By the results of \cite{nostropopp}, we have the formula
\begin{equation}
\mathcal{P} = \frac{1}{2\sum_{i=1}^{d} \alpha_i^2}dx_1 \wedge \ldots \wedge dx_{2d} \wedge dz.
\end{equation}
In particular $\mathcal{P}$ is left-invariant and, up to constant scaling, coincides with the Lebesgue volume of $\mathbb{R}^{2d+1}$. One can check that $\dive_{\mathcal{P}}(X_i) = 0$, hence the sub-Laplacian w.r.t.\ $\mathcal{P}$ is the sum of squares\footnote{This is the case for any sub-Riemannian left-invariant structure on a unimodular Lie group \cite{laplacian}.}:
\begin{equation}
\Delta_{\mathcal{P}} = \sum_{i=1}^{2d} X_{i}^2,
\end{equation}
In this setting, we are able to prove the convergence of the sub-Riemannian random walk with volume sampling.

\begin{theorem}\label{t:limit-contact-Heis}
Let $\mathbb{H}_{2d+1}$ be the Heisenberg group, equipped with a general volume $\omega = e^h\mathcal{P}$. Then $L_{\omega,c}^\eps \to L_{\omega,c}$, where 
\begin{equation}
L_{\omega,c} = \sigma(c) \left(\sum_{i=1}^{2d} X_i^2 + 2c X_i(h)\right) = \sigma(c)\left(\dive_{\omega}\circ \grad + (2c-1) \grad(h)\right),
\end{equation}
and $\sigma(c)$ is a constant (see Remark~\ref{r:constantHeis}). 
\end{theorem}
In particular $L_{\omega,c}$ is essentially self-adjoint in $L^2(M,\omega)$ if and only if $c=1/2$ or $\omega = \mathcal{P}$ (i.e. $h$ is constant). The proof of the above theorem is omitted, as it is a consequence of the next, more general, result. In the general case, the picture is different, and quite surprising, since not even the principal symbol is the expected one.

\begin{theorem}\label{t:limit-contact-carnot}
Let $(\R^{2d+1},\distr,\g)$ be a contact Carnot group, equipped with a general volume $\omega = e^h\mathcal{P}$ and let $c \in (0,1]$. Then $L_{\omega,c}^\eps \to L_{\omega,c}$, where
\begin{equation}
L_{\omega,c} = \sum_{i=1}^d \sigma_{i}(c) \left( X_{2i-1}^2 + X_{2i}^2\right) + 2c \sum_{i=1}^d \sigma_i(c) \left(X_{2i-1}(h)X_{2i-1} + X_{2i}(h)X_{2i}\right),
\end{equation}
where $\sigma_1(c),\ldots,\sigma_d(c) \in \R$ are
\begin{equation}
\sigma_{i}(c) := \frac{c d }{(d+1)\sum_{i=1}^d \int_{-\infty}^{+\infty} |g_i(y)| dy}  \sum_{\ell=1}^d (1+\delta_{\ell i}) \int_{-\infty}^{+\infty}  |g_\ell(c p_z)|\frac{\sin(\tfrac{\alpha_i p_z}{2})^2}{(\alpha_i p_z/2)^2}dp_z,
\end{equation}
and, for $i=1,\ldots,d$
\begin{equation}
g_i(y)= \left( \prod_{j\neq i} \sin\left(\tfrac{\alpha_j y}{2}\right)\right)^2 \frac{\sin\left(\tfrac{\alpha_i y}{2}\right) \left(\tfrac{\alpha_i y}{2} \cos\left(\tfrac{\alpha_i y}{2}\right)- \sin\left(\tfrac{\alpha_i y}{2}\right)\right) }{(y/2)^{2d+2}}.
\end{equation}
Moreover, $P^\eps_{\omega,c} \to P_{\omega,c}$ weakly, where $P_{\omega,c}$ is the law of the process associated with $L_{\omega,c}$.
\end{theorem}

\begin{rmk}[Heisenberg]\label{r:constantHeis}
If $\alpha_1=\ldots=\alpha_d = 1$, the functions $g_i = g$ are equal and
\begin{equation}
\sigma(c) := \sigma_i(c) = \frac{c}{\int_{\R} |g(y)| dy} \int_{\R} |g(c y)| \frac{\sin(y/2)^2}{(y/2)^2},
\end{equation}
does not depend on $i$. In general, however, $\sigma_i \neq \sigma_j$ (see Figure~\ref{f:coeff}).
\end{rmk}
\begin{figure}[h]
\includegraphics[scale=0.4]{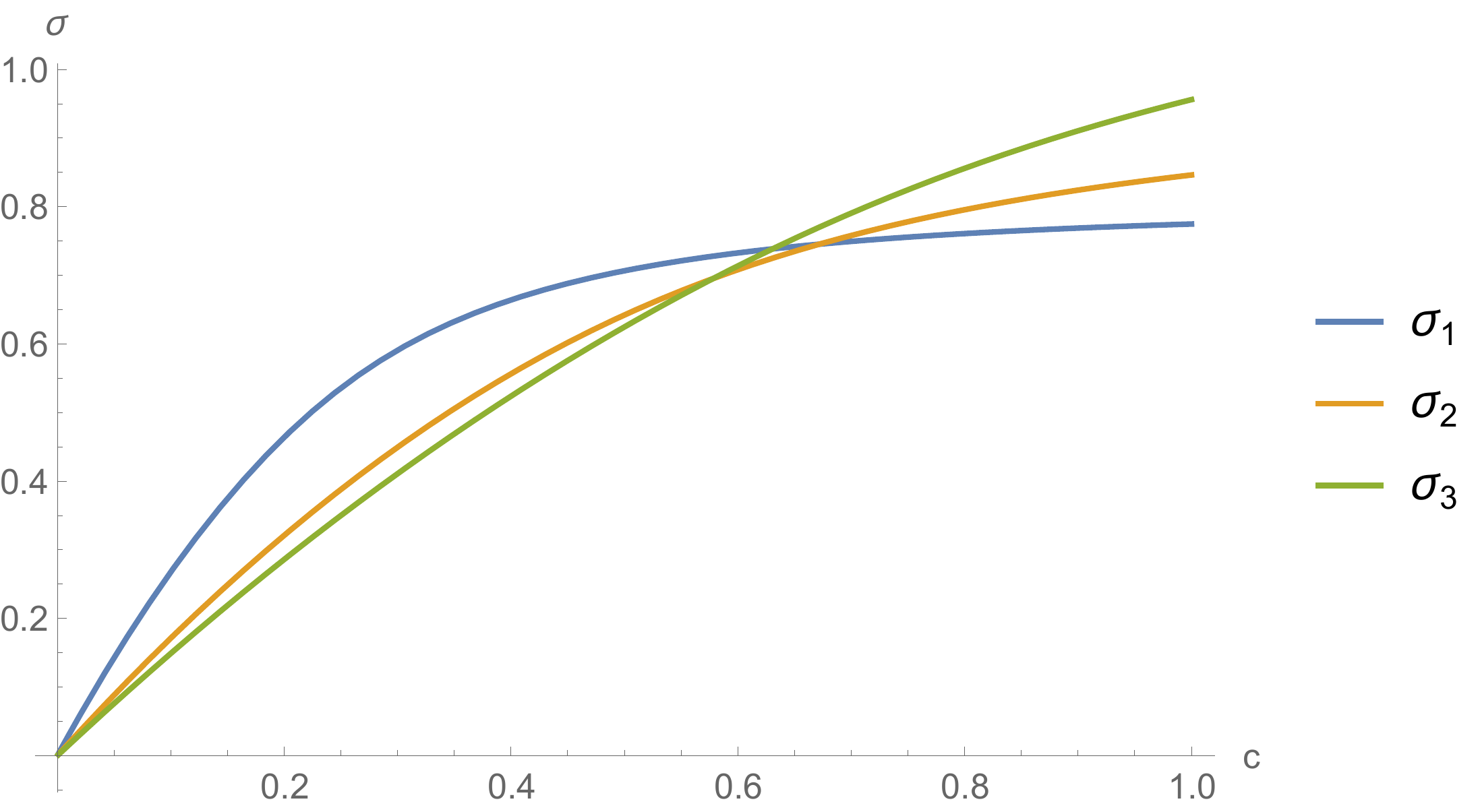}
\caption{Plots of $\sigma_i(c)$ for $d=3$ and $\alpha_1=1$, $\alpha_2 = 2$ and $\alpha_3 =3$.}\label{f:coeff}
\end{figure}

\subsubsection{An intrinsic formula}\label{Sect:IntrinsicFormula}

We rewrite the operator of Theorem \ref{t:limit-contact-carnot} in an intrinsic form. Define a new contact Carnot structure ($\R^{2d+1},\distr,\g'$) on the same distribution, by defining
\begin{equation}\label{eq:primedfields}
X'_{2i-1}:= \sqrt{\sigma_i(c)} X_{2i-1}, \qquad X'_{2i}:=\sqrt{\sigma_i(c)} X_{2i}, \qquad i=1,\ldots,d,
\end{equation}
to be a new orthonormal frame. Observe that this construction does not depend on the choice of $\omega$. Let $\grad$ and $\grad'$ denote the horizontal gradients w.r.t.\ the sub-Riemannian metrics $\g$ and $\g'$, respectively. Then the following is a direct consequence of Theorem \ref{t:limit-contact-carnot} and the definition of this ``primed'' structure.
\begin{cor}\label{cor:intrinsicformula}
The limit operator $L_{\omega,c}$ of Theorem \ref{t:limit-contact-carnot} is
\begin{equation}
L_{\omega,c} = \dive_\omega\circ \grad' + (2c-1) \grad'(h),
\end{equation}
where $\grad'(h) =\sum_{i=1}^{2d} X'_i(h)X'_i$ is understood as a derivation.
\end{cor}
Again $L_{\omega,c}$ is essentially self-adjoint in $L^2(M,\omega)$ if and only if $c=1/2$ or $\omega = \mathcal{P}$ (i.e.\ $h$ is constant). In both cases it is a ``divergence of the gradient'', i.e.\ a well-defined, intrinsic and symmetric operator but, surprisingly, not the expected one. In particular, the behavior of associated heat kernel (e.g.\ its asymptotics) depends not on the original sub-Riemannian metric $\g$, but on the new one $\g'$.

\subsubsection{On the symbol}\label{Sect:Symbol}
We recall that the (principal) symbol of a smooth differential operator $D$ on a smooth manifold $M$ can be seen as a function $\Sigma(D) : T^*M \to \R$. The symbol associated with the sub-Riemannian geodesic random walk with volume sampling is
\begin{equation}
\Sigma(L_{\omega,c})(\lambda) = \sum_{i=1}^d \sigma_i(c) (\langle \lambda, X_{2i-1}\rangle^2 + \langle \lambda, X_{2i}\rangle^2), \qquad \lambda \in T^*M,
\end{equation}
and does not depend on $\omega$. On the other hand, the principal symbol of $\Delta_\omega$ is
\begin{equation}
\Sigma(\Delta_{\omega})(\lambda) = \sum_{i=1}^{2d} \langle \lambda,X_i\rangle^2 = 2H(\lambda), \qquad \lambda \in T^*M.
\end{equation}
In general, the two symbols are different, for any value of the sampling ratio $c>0$. The reason behind this discrepancy is that the family of operators $L_{\omega,c}^\eps$ keeps track of the different eigenspace associated with the generically different singular values $\alpha_i \neq \alpha_j$, through the Jacobian determinant of the exponential map.

\subsection{Alternative construction for the sub-Riemannian random walk}\label{s:altconstr}

An alternative construction of the sub-Riemannian random walk of Section~\ref{s:SR-RW} is the following. For any fixed step length $\eps >0$, one follows only minimizing geodesics segments, that is $\lambda \in \cyl^\eps_q$, as defined in \eqref{eq:cotingj}. In other words, for $\eps >0$, and $c \in (0,1]$, we consider the restriction of $\mu_q^{c\eps}$ to $\cyl_q^\eps$ (which we normalize in such a way that $\int_{\cyl_q^\eps} \mu_q^{c\eps} = 1$).
\begin{rmk}
In the the original construction the endpoints of the first step of the walk lie on the \emph{front of radius $\eps$ centered at $q$}, that is the set $F_q(\eps)=\exp_q(\eps;\cyl_q)$. With this alternative construction, the endpoints lie on the \emph{metric sphere of radius $\eps$ centered at $q$}, that is the set $S_q(\eps)=\exp_q(\eps;\cyl_q^\eps)$. 
\end{rmk}
\begin{rmk}
In the Riemannian setting, locally, for $\eps>0$ sufficiently small, all geodesics starting from $q$ are optimal at least up to length $\eps$, and the two constructions coincide.
\end{rmk}

This construction requires the explicit knowledge of $\cyl_q^\eps$, which is known for contact Carnot groups \cite{ABB-Hausdorff}. We obtain the following convergence result, whose proof is similar to that of Theorem~\ref{t:limit-contact-carnot}, and thus omitted.
\begin{theorem}\label{t:limit-contact-carnot-alternative} 
Consider the geodesic sub-Riemannian random walk with volume sampling, with volume $\omega$ and ratio $c$, defined according to the alternative construction. Then the statement of Theorem~\ref{t:limit-contact-carnot} holds, replacing the constants $\sigma_i(c)\in \R$ with
\begin{equation}
\sigma_{i}^{alt}(c) := \frac{c d  }{(d+1)\sum_{j=1}^d \int_{-2\pi/\alpha_d c}^{2\pi/\alpha_dc} |g_j(y)| dy} \sum_{\ell=1}^d (1+\delta_{\ell i}) \int_{-2\pi/\alpha_d}^{2\pi/\alpha_d}  |g_\ell(c p_z)|\frac{\sin(\tfrac{\alpha_i p_z}{2})^2}{(\alpha_i p_z/2)^2} dp_z,
\end{equation}
for $i=1,\ldots,d$. We call $L_{\omega,c}^{alt}$ the corresponding operator.
\end{theorem}
\begin{rmk}[The case $c = 0$]
In the Riemannian setting the case $c=0$ represents the geodesic random walk with no volume sampling of Section~\ref{s:ito-intro}. In fact, by Theorem~\ref{t:limit-Riemannian},
\begin{equation}
L_{\omega,0} = \lim_{c \to 0^+} L_{\omega,c} = \dive_{\mathcal{R}} \circ \grad, \qquad \text{(Riemannian geodesic RW)},
\end{equation}
is the Laplace-Beltrami operator, for any choice of $\omega$. In the sub-Riemannian setting the case $c=0$ is not defined, but we can still consider the limit for $c \to 0^+$ of the operator. In the original construction, $\lim_{c\to 0^+} \sigma_i(c) = 0$ and by Theorem~\ref{t:limit-contact-carnot} we have:
\begin{equation}
\lim_{c \to 0^+} L_{\omega,c} = 0, \qquad \text{(sub-Riemannian geodesic RW)}.
\end{equation} 
For the alternative sub-Riemannian geodesic random walk discussed above, we have:
\begin{equation}
\lim_{c\to 0^+} \sigma_i^{alt}(c) = \frac{d}{4\pi(d+1)} \left(1 + \frac{\alpha_i^2}{\sum_{\ell=1}^d \alpha_\ell^2}\right) \int_{-2\pi}^{2\pi} \sinc\left(\frac{\alpha_i x}{2\alpha_d}\right)^2 dx, \qquad \forall i = 1,\ldots,d.
\end{equation}
As in Section~\ref{Sect:IntrinsicFormula}, we can define a new metric $\g''$, on the same distribution, such that
\begin{equation}
X_{2i-1}'':= \sqrt{\sigma_i^{alt}(0)} X_{2i-1}, \qquad X_{2i}'':= \sqrt{\sigma_i^{alt}(0)} X_{2i}, \qquad i=1,\ldots,d
\end{equation}
are a global orthonormal frame, where $\sigma_i^{alt}(0):=\lim_{c \to 0^+} \sigma_i^{alt}(c) > 0$. Then, by Theorem~\ref{t:limit-contact-carnot-alternative} we obtain a formula similar to the one of Corollary~\ref{cor:intrinsicformula}:
\begin{equation}
L_{\omega,0}^{alt}:=\lim_{c \to 0^+} L_{\omega,c}^{alt} = \dive_{\mathcal{P}} \circ \grad'', \qquad \text{(alternative sub-Riemannian geodesic RW)}.
\end{equation}
where $\grad''$ is the horizontal gradient computed w.r.t. $\g''$. Unless all the $\alpha_i$ are equal, in general $\sigma_i^{alt}(0) \neq \sigma_j^{alt}(0)$ and $\grad''$ is not proportional to $\grad$.
Notice that $L_{\omega,0}^{alt}$ is a non-zero operator, , symmetric w.r.t.\ Popp volume, and it does not depend on the choice of the initial volume $\omega$. This makes $L_{\omega,0}^{alt}$ (and the corresponding diffusion) an intriguing candidate for an intrinsic sub-Laplacian (and an intrinsic Brownian motion) for contact Carnot groups. 

For the Heisenberg group $\mathbb{H}_{2d+1}$, where $\alpha_i = 1$ for all $i$, by Theorem~\ref{t:limit-contact-Heis} we have:
\begin{equation}
L_{\omega,0}^{alt} = \sigma^{alt}(0) \dive_{\mathcal{P}} \circ \grad, \quad \text{where} \quad \sigma^{alt}(0) = \frac{1}{4\pi}\int_{-2\pi}^{2\pi} \sinc(x)^2 dx.
\end{equation}
\end{rmk}

\begin{rmk}[Signed measures]
A further alternative construction is one in which we remove the absolute value in the definition~\ref{d:mesIto} of $\mu_q^\eps$ on $\cyl_q$. In this case we lose the probabilistic interpretation, and we deal with a signed measure; still, we have an analogue of Theorem~\ref{t:limit-contact-carnot} for the operators themselves, replacing the constants $\sigma_1(c),\ldots,\sigma_d(c)$ with
\begin{equation}
\widetilde{\sigma}_i(c)= \frac{c d}{(d+1)\sum_{j=1}^d \int_{-\infty}^{+\infty} g_j(y) dy} \sum_{\ell=1}^d (1+\delta_{\ell i}) \int_{-\infty}^{+\infty}  g_\ell(c p_z)\frac{\sin(\tfrac{\alpha_i p_z}{2})^2}{(\alpha_i p_z/2)^2} dp_z.
\end{equation}
We observe the same qualitative behavior as the initial construction highlighted in Section~\ref{Sect:IntrinsicFormula} and~\ref{Sect:Symbol}.
\end{rmk}

\subsection{The 3D Heisenberg group}

We give more details for the sub-Riemannian geodesic random walk in the 3D Heisenberg group. This is a contact Carnot group with $d=1$ and $\alpha_1=1$. The identity of the group is $(x,z) = 0$. In coordinates $(p_x,p_z) \in T_0^*M$ we have
\begin{align}
\cyl_0 & = \{(p_x,p_z) \in \R^2 \times \R \mid \|p_x\|^2 = 1\},\\
\cyl_0^\eps & = \{(p_x,p_z) \in \R^2 \times \R \mid \|p_x\|^2 = 1, \quad |p_z|\leq 2\pi/\eps \}.
\end{align}
see \cite{ABB-Hausdorff}. For instance, we set $\omega$ equal to the Lebesgue volume. From the proof of Theorem~\ref{t:limit-contact-carnot}, we obtain, in cylindrical coordinates $(\theta,p_z) \in \mathbb{S}^1 \times \R \simeq T_0^*M$
\begin{equation}
\mu_0^{c\eps} = \begin{dcases}
\frac{c\eps |g(c \eps p_z)|}{2\pi\int_{-\infty}^{\infty} |g(y)|dy}  d\theta \wedge dp_z  & \text{original construction},  \\
\frac{c\eps |g(c \eps p_z)|}{2\pi\int_{-2\pi c}^{2\pi c}  |g(y)| dy}  d\theta \wedge dp_z & \text{alternative construction}, \\
\end{dcases}
\end{equation}
where
\begin{equation}
g(y)= \frac{\sin\left(\tfrac{y}{2}\right) \left(\tfrac{y}{2} \cos\left(\tfrac{y}{2}\right)- \sin\left(\tfrac{y}{2}\right)\right) }{(y/2)^{4}}.
\end{equation}
The normalization is determined by the conditions
\begin{equation}
\begin{cases}
\int_{\cylsmall_0} |\mu_0^{c\eps}| = 1  & \text{original construction}, \\
\int_{\cylsmall_0^\eps} |\mu_0^{c\eps}| = 1 & \text{alternative construction}.
\end{cases}
\end{equation}
The density corresponding to $\mu_0^{c\eps}$, in coordinates $(p_x,p_z)$, depends only on $p_z$. For any fixed $c>0$, the density (for either construction) spreads out  as $\eps \to 0$, and thus the probability to follow a geodesic with large $p_z$ increases (see Fig.~\ref{f:spread}).

\begin{figure}
\includegraphics[scale=1]{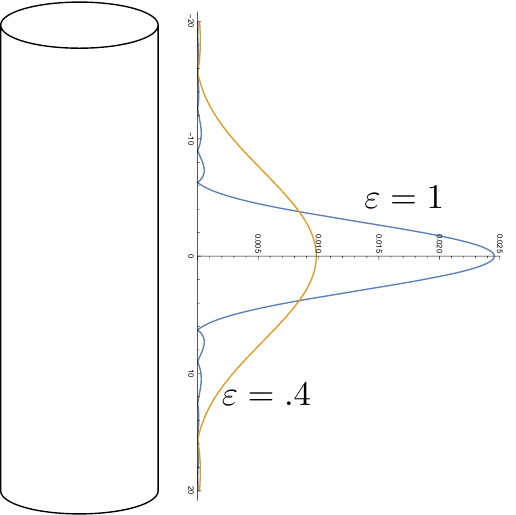}
\caption[meas]{Measures on $\cyl$ for $c=1$ in the Heisenberg group $\mathbb{H}_3$ for the original construction. Each zero corresponds to a conjugate point.}\label{f:spread}
\end{figure}


\section{Flow random walks}\label{s:RW-flow}

The main difficulties to deal with in the convergence of the sub-Riemannian geodesic random walk with volume sampling scheme were related to the non-compactness of $\cyl_q$, and the lack of a general asymptotics for $\mu_q^\eps$. To overcome these difficulties, we discuss a different class of walks. This approach is inspired by the classical integration of a Stratonovich SDE, and can be implemented on Riemannian and sub-Riemannian structures alike (the only requirement being a set of vector fields $V_1,\ldots,V_k$ on a smooth manifold $M$, and a volume $\omega$ for volume sampling).

\subsection{Stratonovich SDEs via flow random walks}  \label{s:strato-intro}

Let $M$ be a smooth $n$-dimensional manifold, and let $V_1,\ldots,V_k$ be smooth vector fields on $M$. Since SDEs are fundamentally local objects (at least in the case of smooth vector fields, where the SDE has a unique, and thus Markov, solution), we do not worry about the global behavior of the $V_i$, and thus we assume, without loss of generality, that the flow along any vector field $V=\beta_1V_1+\beta_2V_2+\cdots+\beta_kV_k$ for any constants $\beta_i$ exists for all time. Further, we can assume that there exists a Riemannian metric $\g$ on $M$ such that the $V_i$ all have bounded norm.

We consider the Stratonovich SDE
\begin{equation}\label{Eqn:StratSDE}
dq_t = \sum_{i=1}^k V_i\lp q_t\rp \circ d(\sqrt{2} w_t^i) , \qquad q_0=q,
\end{equation}
for some $q\in M$, where $w_t^1,\ldots,w_t^k$ are independent, one-dimensional Brownian motions. We recall that solving this SDE is essentially equivalent to solving the martingale problem for the operator $\sum_{i=1}^k V_i^2$. (See \cite[Chapter 5]{KaratzasShreve} for the precise relationship between solutions to SDEs and solutions to martingale problems, although in this case, because of strong uniqueness of the solution to the \eqref{Eqn:StratSDE}, the situation is relatively simple.) We also assume that the solution to \eqref{Eqn:StratSDE}, which we denote $q^0_t$, does not explode.

The sequence of random walks which we associate to \eqref{Eqn:StratSDE} is as follows. We take $\eps>0$. Consider the $k$-dimensional vector space of all linear combinations $\beta_1V_1+\beta_2V_2+\cdots+\beta_kV_k$. Then we can naturally identify $\mS^{k-1}$ with the set $\sum_{i=1}^k \beta_i^2=1$, and thus choose a $k$-tuple $(\beta_1,\ldots,\beta_k)$ from the sphere according to the uniform probability measure. This gives a random linear combination $V=\beta_1V_1+\beta_2V_2+\cdots+\beta_kV_k$. Now, starting from $q$, we flow along the vector field $\frac{2k}{\eps} V$ for time $\dt= \eps^2/(2k)$, traveling a curve of length $\eps\|V\|_{\g}$. This determines the first step of a random walk (and the measure $\Pi^{\eps}_q$). Determining each additional step in the same way produces a family of random walks $q^{\eps}_t$, that we call the \emph{flow random walk} at scale $\eps$ associated with the SDE \eqref{Eqn:StratSDE}.

We associate to each process $q^{\eps}_t$ and $q^0_t$ the corresponding probability measures $P^{\eps}$ and $P^0$ on $\Omega(M)$. The operator induced by the walks converges to the sum-of-squares operator $\sum_{i=1}^k V_i^2$, uniformly on compact sets, by smoothness. Then, by Theorem \ref{t:convergence}, the measures $P^{\eps} \to  P^0$ weakly as $\eps\rightarrow 0$. Note that since this holds for any metric $\g$ as described above, this is really a statement about processes on $M$ as a smooth manifold, and the occurrence of $\g$ is just an artifact of the formalism of Theorem~\ref{t:convergence}. Also, we again note that in this construction, it would be possible to allow $k>n$.

The relationship of Stratonovich integration to ODEs, and thus flows of vector fields, is not new. Approximating the solution to a Stratonovich SDE by an ODE driven by an approximation to Brownian motion is considered in \cite{Doss} and \cite{Sussmann}. Here, we have tried to give a simple, random walk approach emphasizing the geometry of the situation. Nonetheless, because $M$ is locally diffeomorphic to $\bR^n$ (or a ball around the origin in $\bR^n$, depending on one's preferences) and the entire construction is preserved by diffeomorphism, there is nothing particularly geometric about the above, except perhaps the observation that the construction is coordinate independent.

\subsection{Volume sampling through the flow}

The random walk defined in the previous section, which depends only on the choice of $k$ smooth vector fields $V_1,\ldots,V_k$ fits in the general class of walks of Section~\ref{s:convergence}. Moreover, the construction can be generalized to include a volume sampling technique, as we now describe.

Here $V_1,\ldots,V_k$ are a fixed set of global orthonormal fields of a complete (sub-)Rieman\-nian structure, and for this reason we rename them $X_1,\ldots,X_k$. We will discuss in which cases the limit diffusion does not depend on this choice. Notice that, as a consequence of our assumption on completeness of the (sub-)Riemannian structure, any linear combination of the $X_i$'s with constant coefficients is complete.

\begin{rmk}
If $TM$ is not trivial, clearly such a global frame does not exist. To overcome this difficulty, one can consider a locally finite cover $\{U_i\}_{i \in I}$, each one equipped with a preferred local orthonormal frame. For each $q \in M$, there exists a finite set of indices $I_q$ such that $q \cap U_i \neq \emptyset$. Hence, one can easily generalize the forthcoming construction by choosing with uniform probability one of the finite number of available local orthonormal frames available at $q$. Another possibility is to consider an overdetermined set $X_1,\ldots,X_N$ of global vector fields generating the same (sub-)Rieman\-nian structure, as explained in \cite[Sec. 3.1.4]{nostrolibro}. Either choice leads to equivalent random walks hence, for simplicity, we restrict in the following to the case of trivial $TM$.
\end{rmk}

\begin{definition}
For any $q \in M$, and $\eps>0$, the \emph{endpoint map} $E_{q,\eps} : \R^k \to M$ gives the point $E_{q,\eps}(u)$ at time $\eps$ of integral curve of the vector field $X_u:=\sum_{i=1}^k u_i X_i$ starting from $q \in M$. Moreover, let $S_{q,\eps}:=E_{q,\eps}(\mathbb{S}^{k-1})$.
\end{definition}
\begin{rmk}
For small $\eps\geq 0$, $E_{q,\eps} : \mathbb{S}^{k-1} \to S_{q,\eps}$ is a diffeomorphism, and for any unit $u \in \mathbb{S}^{k-1}$, note that $\gamma_u(\eps + \tau):=E_{q,\eps+\tau}(u)$ is a segment of the flow line transverse to $S_{q,\eps}$.
\end{rmk}
The next step is to induce a probability measure $\mu_{q}^\eps$ on $\mathbb{S}^{k-1}$ via volume sampling through the endpoint map. We start with the Riemannian case.

\subsection{Flow random walks with volume sampling in the Riemannian setting}

In this case $k=n$, and the specification of the volume sampling scheme is quite natural.

\begin{definition}\label{d:fakeRiem}
Let $(M,\g)$ be a Riemannian manifold. For any $q \in M$ and $\eps >0$, we define the family of densities on $\mu_q^\eps$ on $\mathbb{S}^{n-1}$
\begin{equation}
\mu_q^\eps(u):= \frac{1}{N(q,\eps)} \left\lvert E_{q,\eps}^* \circ \iota_{\dot\gamma_u(\eps)} \omega)(u)\right\rvert, \qquad \forall u \in \mathbb{S}^{n-1}.
\end{equation}
where $N(q,\eps)$ is fixed by the condition $\int_{\mathbb{S}^{n-1}}\mu_q^\eps = 1$. For $\eps =0$, we set $\mu_q^0$ the standard normalized density on $\mathbb{S}^{n-1}$.
\end{definition}
Then, we define a random walk by choosing $u \in \mathbb{S}^{k-1}$ according to $\mu_q^\eps$, and following the corresponding integral curve. That is, for $\eps > 0$
\begin{equation}\label{eq:process-fakeriem}
r_{(i+1)\dt,c}^\eps:=E_{r_{i\dt,c}^\eps,\eps}(u), \qquad u \in \mathbb{S}^{n-1} \text{ chosen with probability } \mu_q^{c\eps},
\end{equation}
where we have also introduced the parameter $c \in [0,1]$ for the volume sampling. This class of walks includes the one described in the previous section (by setting $c=0$).

Let $P_{\omega,c}^\eps$ be the probability measure on the space of continuous paths on $M$ associated with $r_{t,c}^\eps$ and consider the associated family of operators that, in this case, is
\begin{equation}\label{eq:operator-fakeriem}
(L_{\omega,c}^\eps \phi)(q) := \frac{1}{ \dt}\int_{\mathbb{S}^{n-1}} [\phi(E_{q,\eps}(u))-\phi(q)] \mu_q^{c\eps}(u), \qquad \forall q \in M,
\end{equation}
for any $\phi \in C^\infty(M)$.
\begin{theorem}\label{t:limit-Riemannian-fake}
Let $(M,\g)$ be a complete Riemannian manifold and $X_1,\ldots,X_n$ be a global set of orthonormal vector fields. Let $c \in [0,1]$ and $\omega = e^h \mathcal{R}$ be a fixed volume on $M$, for some $h \in C^\infty(M)$. Then $L_{c,\omega}^\eps \to L_{c,\omega}$, where
\begin{align}
L_{\omega,c} = \Delta_{\omega} + c \grad(h) + (c-1) \sum_{i=1}^n \div_\omega(X_i)X_i.
\end{align}
Moreover $P_{\omega,c}^\eps \to P_{\omega,c}$ weakly, where $P_{\omega,c}$ is the law of the process associated with $L_{\omega,c}$ (which we assume does not explode).
\end{theorem}

The limiting operator is not intrinsic in general, as it clearly depends on the choice of the orthonormal frame. However, thanks to this explicit formula, we have the following.
\begin{cor}\label{c:corflow1}
The operator $L_{\omega,c}$ does not depend on the choice of the orthonormal frame if and only if $c=1$. In this case
\begin{align}
L_{\omega,1} = \Delta_{\omega} +  \grad(h)  = \Delta_{e^h \omega} = \Delta_{e^{2h} \mathcal{R}} .
\end{align}
\end{cor}
Even though $L_{\omega,1}$ is intrinsic and depends only on the Riemannian structure and the volume $\omega$, it is not symmetric in $L^2(M,\omega)$ unless we choose $h$ to be constant. This selects a preferred volume $\omega = \mathcal{R}$, up to a proportionality constant.
\begin{cor}\label{c:corflow2}
The operator $L_{\omega,c}$ with domain $C^\infty_c(M)$ is essentially self-adjoint in $L^2(M, \omega)$ if and only if $c=1$ and $\omega$ is proportional to the Riemannian volume.
\end{cor}
On the other hand, by setting $c=0$, we recover the ``sum of squares'' generator of the solution of the Stratonovich SDE \eqref{Eqn:StratSDE}.
\begin{cor}\label{c:corflow3}
The operator $L_{\omega,0}$ depends on the choice of the vector fields $X_1,\ldots,X_n$, but not on the choice of the volume $\omega$, in particular
\begin{align}
L_{\omega,0} = \sum_{i=1}^n X_i^2.
\end{align}
\end{cor}

\subsection{Flow random walks with volume sampling in the sub-Riemannian setting}

To extend the flow random walk construction to the sub-Riemannian setting we need vector fields  $Z_1,\ldots,Z_{n-k}$ on $M$, transverse to $\distr$, in such a way that $\iota_{Z_1,\ldots,Z_{n-k}} \omega$ is a well-defined $k$-form that we can use to induce a measure on $\mathbb{S}^{k-1}$ as in Definition~\ref{d:fakeRiem}.

In general there is no natural choice of these $Z_1,\ldots,Z_{n-k}$. We explain the construction in detail for contact sub-Riemannian structures, where such a natural choice exists. Indeed, this class contains contact Carnot groups.

\subsubsection{Contact sub-Riemannian structures}
A sub-Riemannian structure $(M,\distr,\g)$ is \emph{contact} if there exists a global one-form $\eta$ such that $\distr = \ker \eta$. This forces $\dim(M) = 2d+1$ and $\rank \distr = 2d$, for some $d \geq 1$. Consider the skew-symmetric \emph{contact endomorphism} $J : \Gamma(\distr) \to \Gamma(\distr)$, defined by the relation
\begin{equation}
g(X,JY) = d\eta(X,Y), \qquad  \forall X, Y \in \Gamma(\distr).
\end{equation}
We assume that $J$ is non-degenerate. Multiplying $\eta$ by a non-zero smooth function $f$ gives the same contact structure, with contact endomorphism $f J$. We fix $\eta$ up to sign by taking
\begin{equation}
\tr(JJ^*) = 1.
\end{equation}
The Reeb vector field is defined as the unique vector $X_0$ such that
\begin{equation}
\eta(X_0) = 1, \qquad \iota_{X_0} d\eta = 0.
\end{equation}
In this case the Popp density is the unique density such that $\mathcal{P}(X_0,X_1,\ldots,X_{2d}) =1$ for any orthonormal frame $X_1,\ldots,X_{2d}$ of $\distr$ (see \cite{nostropopp}).

The flow random walk with volume sampling, with volume $\omega$ and sampling ratio $c$, can be implemented as follows.
\begin{definition}
Let $(M,\distr,\g)$ be a contact sub-Riemannian structure with Reeb vector field $X_0$. For any $q \in M$ and $\eps >0$ we define the family of densities $\mu_q^\eps$ on $\mathbb{S}^{k-1}$
\begin{equation}
\mu_q^\eps(u):= \frac{1}{N(q,\eps)} \left\lvert(E_{q,\eps}^* \circ \iota_{X_0,\dot\gamma_u(\eps)} \omega)(u)\right\rvert, \qquad \forall u \in \mathbb{S}^{k-1},
\end{equation}
where $N(q,\eps)$ is fixed by the condition $\int_{\mathbb{S}^{k-1}} \mu_q^\eps = 1$. For $\eps = 0$, we set $\mu_q^0$ to be the standard normalized density on $\mathbb{S}^{k-1}$. 
\end{definition}
We define a random walk $r_{t,c}^\eps$ as in \eqref{eq:process-fakeriem}, with sampling ratio $c \in [0,1]$, and we call the associated family of operators $L_{\omega,c}^\eps$ as in \eqref{eq:operator-fakeriem}, with no risk of confusion.

\begin{theorem}\label{t:limit-contact-fake}
Let $(M,\distr,\g)$ be a complete contact sub-Riemannian manifold and $X_1,\ldots,X_{2d}$ be a global set of orthonormal vector fields. Let $c \in [0,1]$ and $\omega = e^h \mathcal{P}$ be a fixed volume on $M$, for some $h \in C^\infty(M)$. Then $L_{c,\omega}^\eps \to L_{c,\omega}$, where
\begin{align}
L_{\omega,c} = \Delta_{\omega} + c \grad (h) + (c-1) \sum_{i=1}^k \div_\omega(X_i)X_i.
\end{align}
Moreover $P_{\omega,c}^\eps \to P_{\omega,c}$ weakly, where $P_{\omega,c}$ is the law of the process associated with $L_{\omega,c}$ (which we assume does not explode).
\end{theorem}
This construction, in the contact sub-Riemannian case, has the some properties as the Riemannian one, where the Riemannian volume is replaced by Popp one. In particular we have the following analogues of Corollaries~\ref{c:corflow1},~\ref{c:corflow2}, and \ref{c:corflow3}.
\begin{cor}
The operator $L_{\omega,c}$ does not depend on the choice of the orthonormal frame if and only if $c=1$. In this case
\begin{align}
L_{\omega,1}  = \Delta_{\omega} +  \grad_H(h) = \Delta_{e^h \omega}  = \Delta_{e^{2h} \mathcal{P}} .
\end{align}
\end{cor}
\begin{cor}
The operator $L_{\omega,c}$ with domain $C^\infty_c(M)$ is essentially self-adjoint in $L^2(M,\omega)$ if and only if $c=1$ and $\omega$ is proportional to the Popp volume.
\end{cor}
\begin{cor}
The operator $L_{\omega,0}$ depends on the choice of the vector fields $X_1,\ldots,X_k$, but not on the choice of the volume $\omega$, in particular
\begin{align}
L_{\omega,0} = \sum_{i=1}^k X_i^2.
\end{align}
\end{cor}


\section{Proof of the results}\label{s:proofs}

\subsection{Proof of Theorem \ref{t:limit-Riemannian}}

Let $\eps \leq \eps_0$ and $q \in M$. Fix normal coordinates $(x_1,\ldots,x_n)$ on a neighborhood of $q$. In these coordinates, length parametrized geodesics are straight lines $\eps v$, with $v \in S_q M \simeq \mathbb{S}^{n-1}$. In particular
\begin{equation}
\phi(\exp_q(\eps,v)) - \phi(q) = \eps \sum_{i=1}^n v_i  \partial_i \phi + \frac{1}{2} \eps^2 \sum_{i,j=1}^n v_i v_j \partial_{ij}^2 \phi + \eps^3 O_q,
\end{equation}
where all derivatives are computed at $q$. The term $O_q$ denotes a remainder term which is uniformly bounded on any compact set $K \subset M$ by a constant $|O_q| \leq M_K$. When $\omega = \mathcal{R}$ is the Riemannian volume, well-known asymptotics (see, for instance, \cite{GallotLafontaine}) gives
\begin{equation}
\mu_q^{c\eps}(v) = (1 +\eps^2 O_q) d\Omega,
\end{equation}
where $d\Omega$ is the normalized euclidean measure on $\mathbb{S}^{n-1}$. When $\omega = e^h\mathcal{R}$, the above formula is multiplied by a factor $e^{h(\exp_q(c\eps,v))}$, and taking into account the normalization we obtain
\begin{equation}
\mu_q^{c\eps}(v) = \left(1 + \eps c \sum_{i=1}^n v_i\partial_i h  + \eps^2 O_q\right) d\Omega.
\end{equation}
Then, for the operator $L_{\omega,c}^\eps\phi$, evaluated at $q$, we obtain
\begin{align}
(L_{\omega,c}^\eps \phi )|_q& = \frac{2n}{\eps^2} \int_{S_q M} [\phi(\exp_q(\eps,v)) - \phi(q)] \mu_q^{c\eps}(v) \\
& =\frac{2n}{\eps} \sum_{i=1}^n  \partial_i \phi \int_{\mathbb{S}^{n-1}} v_i  d \Omega  + 2n \sum_{i,j=1}^n \left( c \partial_i h  \partial_j\phi +\frac{1}{2} \partial_{ij}^2 \phi \right) \int_{\mathbb{S}^{n-1}} v_j v_i d\Omega   + \eps O_q.
\end{align}
The first integral is zero, while $\int_{\mathbb{S}^{n-1}} v_i v_j d\Omega = \delta_{ij}/n$. Thus we have
\begin{equation}
(L_{\omega,c}^\eps \phi)|_q=\sum_{i=1}^n \partial_{ii}^2 \phi + 2c (\partial_i h)( \partial_i \phi) + \eps O_q.
\end{equation}
The first term is the Laplace-Beltrami operator applied to $\phi$, written in normal coordinates, while the second term coincides with the action of the derivation $2c \grad(h)$ on $\phi$, evaluated at $q$. Since the l.h.s.\ is invariant under change of coordinates, we have $L_{\omega,c}^\eps \to L_{\omega,c}$, where
\begin{equation}
L_{\omega,c} = \Delta_{\mathcal{R}} + 2c \grad(h),
\end{equation}
and the convergence is uniform on compact sets. The alternative forms of the statement follow from the change of volume formula $\Delta_{e^h \omega} = \Delta_{\omega} + \grad(h)$. The convergence $P^\eps_{\omega,c} \to P_{\omega,c}$ follows from Theorem~\ref{t:convergence}. \hfill $\qed$

\subsection{Proof of Theorem \ref{t:limit-contact-carnot}}
We start with the case $h=0$ and $q=0$. The Hamilton equations for a contact Carnot groups are readily solved, and the geodesic with initial covector $(p_x,p_z) \in T_0^*M \simeq \R^{2d} \times \R$ is
\begin{equation}
x(t) = \int_0^t e^{s p_z A} p_x ds, \qquad z(t) = -\frac{1}{2} \int_0^t \dot{x}^*(s) A x(s)  ds.
\end{equation}
It is convenient to split $p_x$ as $p_x = (p_x^1,\ldots,p_x^d)$, with $p_x^i = (p_{x_{2i-1}},p_{x_{2i}}) \in \R^2$ the projection of $p_x$ on the real eigenspace of $A$ corresponding to the singular value $\alpha_i$. We get
\begin{equation}\label{eq:expcorank1}
\exp_0(t;p_x,p_z) = \begin{pmatrix} B(t;\alpha_1 p_z) p^1_x \\
\vdots \\
B(t;\alpha_d p_z) p^d_x \\
\sum_{i=1}^d b(t;\alpha_i p_z) \alpha_i \|p_x^i\|^2
\end{pmatrix},
\end{equation}
where
\begin{equation}
B(t; y ):=\frac{\sin(t y)}{y} \mathbb{I} + \frac{\cos(t y)-1}{y} J, \qquad b(t;y) := \frac{t y - \sin(ty)}{2 y^2}.
\end{equation}
If $p_z =0$, the equations above must be understood in the limit, thus $\exp_0(t;p_x,0) = (tp_x,0)$. The Jacobian determinant is computed in \cite{ABB-Hausdorff} (see also \cite{R-MCP} for the more general case of a corank 1 Carnot group with a notation closer to that of this paper):
\begin{equation}
\det(d_{p_x,p_z}\exp_0(t;\cdot)) = \frac{t^{2d+3}}{4\alpha^2}\sum_{i=1}^d g_i(t p_z) \|p_x^i\|^2,
\end{equation}
where $\alpha = \prod_{i=1}^d \alpha_i$ and
\begin{equation}
g_i(y):= \left( \prod_{j\neq i} \sin\left(\tfrac{\alpha_j y}{2}\right)\right)^2 \frac{\sin\left(\tfrac{\alpha_i y}{2}\right) \left(\tfrac{\alpha_i y}{2} \cos\left(\tfrac{\alpha_i y}{2}\right)- \sin\left(\tfrac{\alpha_i y}{2}\right)\right) }{(y/2)^{2d+2}}.
\end{equation}
\begin{lemma}\label{l:switch}
For any $\lambda \in T_q^*M$ and $t>0$, we have (up to the normalization)
\begin{equation}
(\exp_q(t;\cdot)^* \circ \iota_{\dot\gamma_\lambda(t)}\omega)(\lambda) = \frac{1}{t} \iota_\lambda \circ (\exp_q(t;\cdot)^* \omega)(\lambda).
\end{equation}
\end{lemma}
\begin{proof}
It follows from the homogeneity property $\exp_q(t;\alpha \lambda) = \exp_q(\alpha t; \lambda)$, for all $\alpha \in \R$:
\begin{equation}
\dot{\gamma}_\lambda(t) = \left.\frac{d}{d\tau}\right|_{\tau = 0} \exp_q(t+\tau;\lambda) = \frac{1}{t} \left.\frac{d}{d\tau}\right|_{\tau = 0} \exp_q(t;(1+\tau)\lambda) = \frac{1}{t} d_\lambda \exp_q(t;\cdot) \lambda,
\end{equation}
where we used the standard identification $T_\lambda(T_q^*M) = T_q^*M$. 
\end{proof}
The cylinder is $\cyl_0 = \{(p_x,p_z)\mid \|p_x\|^2 = 1\} \subset T_0^*M$ and $\lambda \simeq p_z \partial_{p_z} + p_x \partial_{p_x}$. The Lebesgue volume is $\mathscr{L}=dx \wedge dz$. By Lemma~\ref{l:switch} and reintroduction of the normalization factor, we obtain that the restriction to $\cyl_0$ of $\mu_0^t$ is
\begin{equation}\label{eq:misuraproof}
\mu_0^t = \frac{1}{N(t)}\sum_{i=1}^d | g_i(p_z t) | \|p_x^i\|^2 | d\Omega \wedge dp_z|,
\end{equation}
where $d\Omega$ is the normalized volume of $\mathbb{S}^{2d-1}$. Observe that each $|g_i| \in L^1(\R)$. Thus
\begin{equation}
N(t)  = \sum_{i=1}^d \int_{\mathbb{S}^{2d-1}} \|p_x^i\|^2 d\Omega\int_{\R} dp_z |g_i(p_z t)| =  \frac{1}{dt} \sum_{i=1}^d \int_{\R} dy |g_i(y)|.
\end{equation}
To compute $\mathbb{E}[\phi(\exp_q(\eps;\lambda))-\phi(q)]$, we can assume $\phi(q) = 0$. Hence
\begin{align}
\int_{\cyl_0}\phi(\exp_0(\eps;\lambda)) \mu_0^{c\eps}(\lambda)  & = \frac{1}{N(c\eps)}\sum_{i=1}^d \int_{\mathbb{S}^{2d-1}}  d\Omega \int_{\R} dp_z |g_i(p_z c \eps)|\|p_x^i\|^2 \phi(\exp_0(\eps;p_x,p_z)) \\
& = \frac{c}{\eps N(\eps)}\sum_{i=1}^d \int_{\mathbb{S}^{2d-1}} \|p_x^i\|^2 d\Omega \int_{\R} d y |g_i(cy)|\phi(\exp_0(\eps;p_x,y/\eps)) \\
& = \frac{c d}{\sum_{i=1}^d \int_{\R} |g_i(y)| dy} \sum_{i=1}^d \int_{\mathbb{S}^{2d-1}} \|p_x^i\|^2 d\Omega \int_{\R} dp_z |g_i(c p_z)|\phi(\exp_0(1;\eps p_x,p_z)), \label{eq:directcompt}
\end{align}
where we used the rescaling property of the exponential map. From~\eqref{eq:expcorank1} we get
\begin{equation}
\exp_0(1;\eps p_x,p_z) = \left(\begin{pmatrix}
B(\alpha_1 p_z)  & & \\
& \ddots & \\
 & & B(\alpha_d p_z) 
\end{pmatrix} \eps p_x, \sum_{i=1}^d b(\alpha_i p_z) \alpha_i \|p_x^i\|^2 \eps^2\right),
\end{equation}
where, with a slight abuse of notation
\begin{equation}
B(y)=\frac{\sin(y)}{y} \mathbb{I} + \frac{\cos(y)-1}{y} J, \qquad b(y) = \frac{y - \sin(y)}{2 y^2}.
\end{equation}
We observe here that
\begin{equation}\label{eq:observation}
B(y) B(y)^* = \frac{\sin(y/2)^2}{(y/2)^2} \mathbb{I}.
\end{equation}
It is convenient to rewrite
\begin{equation}
\exp_0(1;\eps p_x,p_z) = (\mathbf{B}(p_z) \eps p_x, \eps^2 p_x^* \mathbf{b}(p_z) p_x),
\end{equation}
where $\mathbf{B}(p_z)$ is a block-diagonal $2d\times 2d$ matrix, whose blocks are $B(\alpha_i p_z)$, and $\mathbf{b}$ is a $2d\times 2d$ diagonal matrix. Notice that $\exp_0(1;\eps p_x,p_z)$ is contained in the compact metric ball of radius $\eps$. Hence, we have
\begin{equation}\label{eq:directcomptexpansion}
\begin{aligned}
\phi(\exp_0(\eps;tp_x,p_z)) & = (\partial_x \phi)(\mathbf{B}(p_z) \eps p_x) + (\partial_z \phi)  p_x^* \mathbf{b}(p_z) p_x \eps^2 \\
& \quad + \frac{1}{2} \eps^2 (\mathbf{B}(p_z)p_x)^* (\partial^2_{x} \phi) (\mathbf{B}(p_z) p_x) + \eps^3 R_{(p_x,p_z)}(\eps).
\end{aligned}
\end{equation}
All derivatives are computed at $0$. Let $\eps \leq \eps_0$. A lengthy calculation using the explicit form of the remainder (and Hamilton's equations) shows that the remainder term is uniformly bounded (i.e. independently on $\eps$) by an order two polynomial in $p_z$, that is $R_{(p_x,p_z)}(\eps) \leq A + B p_z + Cp_z^2$, where the constants depend on the derivatives of $\phi$	(up to order three) on the compact ball of radius $\eps_0$ centered at the origin. Plugging~\eqref{eq:directcomptexpansion} back in~\eqref{eq:directcompt}, we observe that the term proportional to
\begin{equation}
\int_{\mathbb{S}^{2d-1}} \|p_x^i\|^2 d\Omega \int_{\R} dp_z |g_i(c p_z)| (\partial_x \phi) \mathbf{B}(p_z) \eps p_x
\end{equation}
vanishes, as the integral of any odd-degree monomial in $p_x$ on the sphere is zero. Furthermore, the term proportional to
\begin{equation}
\int_{\mathbb{S}^{2d-1}} \|p_x^i\|^2 d\Omega \int_{\R} dp_z |g_i(c p_z)| (\partial_z \phi)  p_x^* \mathbf{b}(p_z) p_x t^2
\end{equation}
vanishes, as the integrand is an odd function of $p_z$. The last second-order (in $\eps$) term is
\begin{equation}\label{eq:remainingterm}
\frac{cd}{\sum_{i=1}^d \int_{\R} |g_i(y)| dy} \sum_{i=1}^d \int_{\mathbb{S}^{2d-1}} \|p_x^i\|^2 d\Omega \int_{\R} dp_z |g_i(c p_z)| \frac{1}{2} \eps^2 (\mathbf{B}(p_z)p_x)^* (\partial^2_{x} \phi) (\mathbf{B}(p_z) p_x).
\end{equation}
If all the $\alpha_i$ are equal, then all $g_i = g$, and \eqref{eq:remainingterm} the sum $\sum_{i=1}^d \|p_x^i\|^2 = \|p_x\|^2 = 1$ simplifies. In this case we have a simple average of a quadratic form on $\mathbb{S}^{2d-1}$. When the $\alpha_i$ are distinct, we need the following results.
\begin{lemma}[see~\cite{Folland}]\label{l:intpolsphere}
Let $P(x) = x_1^{a_1} \ldots x_n^{a_n}$ a monomial in $\R^n$, with $a_i,\ldots,a_n \in \{0,1,2,\ldots\}$. Set $b_i:= \tfrac{1}{2}(a_i + 1)$ Then
\begin{equation}
\int_{\mathbb{S}^{n-1}} P(x) d\Omega = \frac{\Gamma(n/2)}{2 \pi^{n/2}} \begin{cases}
0 & \text{if some $a_j$ is odd}, \\
\frac{2 \Gamma(b_1)\Gamma(b_2)\cdots \Gamma(b_n)}{\Gamma(b_1+ b_2 + \cdots + b_n)} & \text{if all $a_j$ are even},
\end{cases}
\end{equation}
where $d\Omega$ is the normalized measure on the sphere $\mathbb{S}^{n-1} \subset \R^n$. 
\end{lemma}
\begin{lemma}\label{l:prodquad}
Let $Q(x) = x^* Q x$ and $R(x) = x^* R x$ be two quadratic forms on $\R^n$, such that $QR = RQ$. Then
\begin{equation}
\int_{\mathbb{S}^{n-1}} Q(x)R(x) d\Omega = \frac{2\tr(QR) + \tr(Q) \tr(R)}{n(n+2)}.
\end{equation}
If $R=\mathbb{I}$, we recover the usual formula $\int_{\mathbb{S}^{n-1}} Q d \Omega = \tfrac{1}{n}\tr(Q)$.
\end{lemma}
\begin{proof}
Up to an orthogonal transformation, we can assume that $Q$ and $R$ are diagonal. Hence (for brevity we omit the domain of integration and the measure),
\begin{equation}
\int Q(x)R(x) = \sum_{i,j=1}^n Q_{ii} R_{jj} \int x_i^2 x_j^2.
\end{equation}
By Lemma~\ref{l:intpolsphere}, we have
\begin{equation}
\int x_i^2 x_j^2 = \begin{cases}
\frac{3}{n(n+2)} & i = j, \\
\frac{1}{n(n+2)} & i \neq j.
\end{cases}
\end{equation}
Thus
\begin{align}
\int Q(x)R(x) & = \sum_{i,j=1}^n Q_{ii} R_{jj} \int x_i^2 x_j^2 \left( \delta_{i j} + (1-\delta_{i j})\right) \\
& = \frac{1}{n(n+2)}\sum_{i,j}^n Q_{ii} R_{jj} (3\delta_{ij} + (1-\delta_{ij}))  = \frac{2\tr(QR) + \tr(Q) \tr(R)}{n(n+2)}.\qedhere
\end{align}
\end{proof}
We can write~\eqref{eq:remainingterm}, as the sum of integrals of products of quadratic forms over $\mathbb{S}^{2d-1}$
\begin{equation}\label{eq:remainingterm2}
\frac{1}{2} \eps^2\frac{cd}{\sum_{i=1}^d \int_{\R}| g_i(y)| dy} \sum_{i=1}^d \int_{\R} dp_z |g_i(c p_z)| \int_{\mathbb{S}^{2d-1}} Q_i(p_x) R(p_x) d\Omega,
\end{equation}
where the quadratic forms are (we omit the explicit dependence on $p_z$)
\begin{equation}
Q_i(p_x) := \|p_x^i\|^2, \qquad R(p_x) := (\mathbf{B}(p_z)p_x)^* (\partial^2_{x} \phi) (\mathbf{B}(p_z) p_x).
\end{equation}
A direct check shows that $Q$ and $R$ are commuting, block diagonal matrices. Thus, applying Lemma~\ref{l:prodquad} to~\eqref{eq:remainingterm2}, we obtain
\begin{multline}
\frac{1}{2} \eps^2\frac{cd}{\sum_{i=1}^d \int_{\R} |g_i(y)| dy} \sum_{i=1}^d \int_{\R} dp_z |g_i(c p_z)| \int_{\mathbb{S}^{2d-1}} Q_i(p_x) R(p_x) d\Omega =\\
= \frac{1}{2} \eps^2\frac{c d}{\sum_{i=1}^d \int_{\R} |g_i(y)| dy} \sum_{i=1}^d  \int_{\R} dp_z |g_i(c p_z)| \frac{2\tr(Q_i R) + \tr(Q_i)\tr(R)}{2d(2d+2)}. \label{eq:remainingterm3}
\end{multline}
Observe that $\tr(Q_i) = 2$, and $\sum_{\ell=1}^d Q_\ell = \mathbb{I}$. Therefore we rewrite~\eqref{eq:remainingterm3} as
\begin{equation}\label{eq:remainingterm4}
\eps^2\frac{c}{\sum_{i=1}^d \int_{\R} |g_i(y)| dy} \sum_{i,\ell=1}^d  \int_{\R} dp_z |g_i(c p_z)| \frac{(1+\delta_{i\ell})\tr(Q_\ell R)}{4(d+1)}.
\end{equation}
To compute $\tr(Q_\ell R)$ denote, for $\ell=1,\ldots,d$
\begin{equation}
D^2_\ell \phi:=\begin{pmatrix}
\partial^2_{x_{2\ell-1}}\phi & \partial_{x_{2\ell-1}}\partial_{x_{2\ell}}\phi  \\
\partial_{x_{2\ell}} \partial_{x_{2\ell-1}}\phi & \partial^2_{x_{2\ell}}\phi 
\end{pmatrix}, \qquad B_\ell:= B(\alpha_\ell p_z).
\end{equation}
We thus obtain
\begin{equation}
\tr(Q_\ell R ) =  \tr(B_\ell^* (D^2_\ell \phi) B_\ell) = \trace(B_\ell B_\ell^* (D^2_\ell \phi)) = \frac{\sin(\tfrac{\alpha_\ell p_z}{2})^2}{(\alpha_\ell p_z/2)^2} (\partial^2_{x_{2\ell-1}} \phi + \partial^2_{x_{2\ell}} \phi),
\end{equation}
where we used~\eqref{eq:observation}. Thus~\eqref{eq:remainingterm4} becomes
\begin{equation}
\frac{\eps^2}{4d} \sum_{i =1}^d \sigma_{i}(c)(\partial_{x_{2i-1}} \phi + \partial_{x_{2i}} \phi),
\end{equation}
where the constants $\sigma_{i}(c)$ are as in the statement of Theorem~\ref{t:convergence}.
Taking in account the remainder term as well, we obtain
\begin{equation}
\frac{4d}{\eps^2} \int_{\cyl_0}\phi(\exp_0(\eps;p_x,p_z)) \mu_0^{c\eps}(p_x,p_z)  = \sum_{i=1}^d \sigma_i(c) (\partial^2_{x_{2i-1}}\phi  + \partial^2_{x_{2i}}\phi )|_0 + 4d \eps O_0,
\end{equation}
where $|O_0| \leq M_0$ is a remainder term that, when $\eps \leq \eps_0$, is bounded by a constant that depends only on the derivatives of $\phi$ in a compact metric ball of radius $\eps_0$ centered in $0$. A straightforward left-invariance argument shows that, for any other $q \in M$
\begin{equation}
\frac{4d}{\eps^2} \int_{\cyl_q}[f(\exp_q(\eps;\lambda))-f(q)]\mu_q^{c \eps}(\lambda) = \sum_{i=1}^d \sigma_i(c) (X^2_{2i-1}\phi  + X^2_{2i}\phi)|_q + 4d \eps O_q(1),
\end{equation}
where $O_q\leq M_q$ is a remainder term bounded by a constant that depends only on the derivatives of $\phi$ in a compact metric ball of radius $\varepsilon_0$ centered in $q$. Thus
\begin{equation}
(L_{c,\mathscr{L}}\phi)|_{q} = \lim_{\eps\to 0} \frac{4d}{\eps^2} \int_{\cyl_q}[\phi(\exp_q(\eps;\lambda))-\phi(q)]\mu_q^{c \eps}(\lambda) = \sum_{i=1}^d \sigma_i(c) (X_{2i-1}\phi  + X_{2i}\phi )|_q,
\end{equation}
and the convergence is uniform on compact sets. This completes the proof for $\omega = \mathscr{L}$. 

Let, instead, $\omega = e^h \mathscr{L}$ for some $h \in C^\infty(M)$. This leads to an extra factor $e^{h(\exp_q(c \eps;\lambda))}$ in front of $\mu_q^{c\eps}(\lambda)$ (up to re-normalization). After a moment of reflection one realizes that
\begin{equation}
(L^\eps_{\omega,c} \phi)|_q = (L^\eps_{\mathscr{L},c} \tilde{\phi})|_q + \eps O_q, \qquad \text{ with } \qquad \tilde\phi = e^{c(h-h(q))}(\phi-\phi(q)).
\end{equation}
This observation yields the general statement, after noticing that
\begin{equation}
X_i^2(\tilde{\phi}) = X_i^2(\phi) + 2c X_i(h) X_i(\phi), \qquad \forall i =1,\ldots,2d,
\end{equation}
where everything is evaluated at the fixed point $q$. \hfill $\qed$

\subsection{Proof of Theorem \ref{t:limit-Riemannian-fake}}

We expand the function $\phi$ along the path $\gamma_u(\eps) = E_{q,\eps}(u)$:
\begin{equation}
\phi(E_{q,\eps}(u)) -\phi(q)= \eps X_u (\phi) + \frac{1}{2} \eps^2 X_u( X_u(\phi)) + O(\eps^3), 
\end{equation}
where everything on the r.h.s.\ is computed at $q$ (as a convention, in the following when the evaluation point is not explicitly displayed, we understand it as evaluation at $q$).
\begin{lemma}\label{l:pullbacknu}
For any one-form $\nu \in T_q^*M$ and any vector $v \in T_u\mathbb{S}^{n-1}$
\begin{equation}
(E^*_{q,\eps} \nu)|_u(v) = \eps \nu(X_v) + \frac{1}{2}\eps^2\nu([X_v,X_u]) + O(\eps^3).
\end{equation}
\end{lemma}
\begin{proof}[Proof of Lemma~\ref{l:pullbacknu}]
Since $u$ is constant, the differential of the endpoint map is
\begin{equation}
d_u E_{q,\eps}(v) = e^{\eps X_u}_* \int_0^\eps e^{-\tau X_u}_* X_v d\tau, \qquad v \in \R^n,
\end{equation}
where $e^{\eps Y}$ is the flow of the field $Y$ (see \cite{nostrolibro}). By definition of the Lie derivative $\mathcal{L}$ we get
\begin{equation}
\begin{aligned}
\frac{d}{d\eps}(E^*_{q,\eps} \nu)|_u(v) & =  \frac{d}{d\eps} (e^{\eps X_u *}\nu)|_q\left(\int_0^\eps e^{-\tau X_u}_* X_v d\tau\right) \\
& = (e^{\eps X_u *}\mathcal{L}_{X_u} \nu)|_q\left(\int_0^\eps e^{-\tau X_u}_* X_v d\tau\right) +  (e^{\eps X_u *}\nu)|_q\left(e^{-\eps X_u}_* X_v \right).
\end{aligned}
\end{equation}
Taking another derivative, and evaluating at $t=0$, we get
\begin{align}
\left.\frac{d^2}{d\eps^2}\right|_{\eps=0}(E^*_{q,\eps} \nu)|_u(v)  & = 2(\mathcal{L}_{X_u} \nu)|_q(X_v) +  \nu|_q(\mathcal{L}_{X_u}(X_v))  = \nu([X_v,X_u]),\\
\left.\frac{d}{d\eps}\right|_{\eps=0}(E^*_{q,\eps} \nu)|_u(v)  & = \nu|_q(X_v). \qedhere
 \end{align}
\end{proof}
\begin{lemma}\label{l:expmeasureriem}
We have the following Taylor expansion for the measure
\begin{equation}
\mu_q^\eps(u) = \left(1+\frac{\eps}{2} \dive_{\mathcal{R}}(X_u) + \eps  X_u(h) + O(\eps^2)\right)d\Omega(u),
\end{equation}
where $d\Omega$ is the normalized Euclidean measure on $\mathbb{S}^{n-1}$.
\end{lemma}
\begin{proof}[Proof of Lemma~\ref{l:expmeasureriem}]
Let $\nu_1,\ldots,\nu_n$ be the dual frame to $X_1,\ldots,X_n$, that is $\nu_i(X_j) =\delta_{ij}$. Since $\omega = e^h \mathcal{R} = e^h \nu_1\wedge\ldots\wedge \nu_n$, we obtain (ignoring normalization factors)
\begin{align}\label{eq:measurefakeriem}
\mu_q^\eps(u) \propto D_q(\eps) e^{h(\gamma_u(\eps))}  d\Omega(u), \qquad u \in \mathbb{S}^{n-1},
\end{align}
where $D_q(\eps)$ is the determinant of the matrix $(E_{q,\eps}^* \nu_i)(e_j)$, for $i,j=1,\ldots,n$. Using Lemma~\ref{l:pullbacknu}, since $X_{e_j} = X_j$, we obtain
\begin{align}
(E_{q,\eps}^* \nu_i)(e_j) = \eps \nu_i(X_j) + \frac{\eps^2}{2}\nu_i([X_j,X_u]) + O(\eps^3),
\end{align}
where everything is computed at $q$. Since $\det(\mathbb{I} + \eps M) = 1+\eps \tr(M) +O(\eps^2)$ for any matrix $M$, we get
\begin{equation}
D_q(\eps)  = \eps^{n}\left(1+ \frac{\eps}{2}\sum_{i=1}^n \nu_i([X_i, X_u]) + O(\eps^2)\right) =  \eps^{n}\left(1+ \frac{\eps}{2} \dive_{\mathcal{R}}(X_u) + O(\eps^2)\right).
\end{equation}
Plugging this in~\eqref{eq:measurefakeriem}, and expanding the function $e^{h(\gamma_u(\eps))}$, we get 
\begin{align}
\mu_q^\eps & \propto \eps^{n}\left(1+\frac{\eps}{2} \dive_{\mathcal{R}}(X_u)+O(\eps^2)\right)e^{h(q)}\left(1+\eps X_u(h) + O(\eps^2)\right) d\Omega(u)  \\
& \propto \eps^{n} e^{h(q)} \left(1+\frac{\eps}{2} \dive_{\mathcal{R}}(X_u)+t X_u(h) + O(\eps^2)\right) d\Omega(u).
\end{align}
Taking in account the normalization (recall that $\int_{\mathbb{S}^{n-1}} X_u  = 0$), we obtain  the result.
\end{proof}
We are ready to compute the expectation value
\begin{multline}
\int_{\mathbb{S}^{n-1}}[\phi(E_{q,\eps}(u))-\phi(q)] \mu_q^{c\eps} = 
\int_{\mathbb{S}^{n-1}}\left[\eps X_u (\phi) + \frac{1}{2} \eps^2 X_u( X_u(\phi)) +O(\eps^3)\right] \times \\
\times \left[1+\frac{c\eps}{2} \dive_{\mathcal{R}}(X_u)+c\eps X_u(h) + O(\eps^2)\right]d\Omega(u).
\end{multline}
Since $\int_{\mathbb{S}^{n-1}} X_u =0$ and $\int_{\mathbb{S}^{n-1}} Q_{ij}u_i u_j = \tr (Q)/n$, we get
\begin{align*}
(L_{\omega,c} \phi)(q) & = \lim_{\eps \to 0^+}\frac{2n}{ \eps^2} \left(\frac{c\eps^2}{2n}\dive_\mathcal{R}(X_i)X_i(\phi) + \frac{c\eps^2}{n} X_i(\phi)X_i(h) + \frac{\eps^2}{2n} X_i^2(\phi) + O(\eps^3)\right)  \\
& = \sum_{i=1}^n X_i^2(\phi) + c \dive_{\mathcal{R}}(X_i) X_i(\phi)+ 2c X_i(\phi)X_i(h).
\end{align*}
We obtain the different forms of the statements using the change of volume formula $\dive_{\omega}(X_i) = \dive_{e^h \mathcal{R}}(X_i) = \dive_\mathcal{R}(X_i) + X_i(h)$. The convergence is uniform on compact sets since the domain of integration $\mathbb{S}^{n-1}$ is compact and all objects are smooth. \hfill $\qed$

\subsection{Proof of Theorem \ref{t:limit-contact-fake}}

The proof follows the same lines as that of Theorem~\ref{t:limit-Riemannian-fake}. The expansion of the function $\phi$ along the path $\gamma_u(\eps) = E_{q,\eps}(u)$ remains unchanged:
\begin{equation}
\phi(E_{q,\eps}(u)) -\phi(q)= \eps X_u (\phi) + \frac{1}{2} \eps^2 X_u( X_u(\phi)) + O(\eps^3). 
\end{equation}
where, this time $X_u = \sum_{i=1}^k u_i X_i$. Also Lemma~\ref{l:pullbacknu} remains unchanged, replacing $n$ with $k$. The following contact version of Lemma~\ref{l:expmeasureriem} also holds. 
\begin{lemma}
We have the following Taylor expansion for the measure
\begin{equation}
\mu_q^\eps(u) = \left(1+\frac{\eps}{2} \dive_{\mathcal{P}}(X_u) + \eps  X_u(h) + O(\eps^2)\right)d\Omega(u),
\end{equation}
where $d\Omega$ is the normalized Euclidean measure on $\mathbb{S}^{k-1}$.
\end{lemma}
\begin{proof}[Proof of the Lemma]
 Since $\omega = e^{h} \mathcal{P} = e^h \nu_0 \wedge \nu_1\wedge \ldots \nu_k$, we have $i_{X_0} \omega = e^h \nu_1\wedge \ldots \wedge \nu_k$. Hence the proof is similar to proof of Lemma~\ref{l:expmeasureriem}, with $n$ replaced by $k$. In fact, up to normalization
\begin{equation}
\mu_q^\eps(u)  \propto (E_{\eps,q}^* \circ \iota_{\dot\gamma_u(\eps),X_0} \omega) = D_q(\eps) e^{h(\gamma_u(\eps))} d\Omega(u), \qquad u \in \mathbb{S}^{k-1},
\end{equation}
where $D_q(\eps)$ is the determinant of the matrix $(E_{q,\eps}^* \nu_i)(X_j)$, for $i,j=1,\ldots,k$. This is a $k\times k$ matrix. With a computation analogous to the one in the proof of Lemma~\ref{l:expmeasureriem}, we obtain $D_q(\eps) = \eps^k(1+\eps\tr (M) +O(\eps^2))$, with
\begin{equation}
\tr (M) = \frac{1}{2} \sum_{i=1}^{k} \nu_i([X_i, X_u])  = \frac{1}{2} \sum_{i,j=1}^{k} u_j  c_{ij}^i  = \frac{1}{2} \sum_{j=1}^k u_j \sum_{i=0}^{k} c_{ij}^i  = \frac{1}{2} \dive_{\mathcal{P}}(X_u),
\end{equation}
where we have been able to complete the sum, including the index $0$ since, in the contact case, $c_{0j}^0 = \eta([X_0,X_j]) = -d\eta(X_0,X_j) = 0$ for all $j=1,\ldots,k$. From here, we conclude the proof as in that of Lemma~\ref{l:expmeasureriem}.
\end{proof}
The computation of the limit operator is analogous to the one in the proof of Theorem~\ref{t:limit-Riemannian-fake}, replacing the Riemannian volume $\mathcal{R}$ with the Popp one $\mathcal{P}$. \hfill $\qed$

\appendix

\section{Volume sampling as Girsanov-type change-of-measure}\label{a:Girsanov}

In both the geodesic and flow random walks defined in Sections~\ref{s:ito-intro} and~\ref{s:strato-intro}, the probability measure used to select the vector $V=\sum\beta_iV_i$ was the uniform probability measure on the unit sphere with respect to the covariance structure of the $w^i_t$ (which gives an inner product on the vector space of such $V$). In the volume sampling scheme we have introduced for the geodesic random walk with respect to an orthonormal frame on a Riemannian manifold (that is, the volume sampling scheme for the isotropic random walk that approximates Brownian motion), the probability measure on the sphere is replaced by a different probability measure, absolutely continuous with respect to the uniform one. In terms of the random walk, the volume-sampled walk is supported on the same set of paths as the original walk, but with a different probability measure, absolutely continuous with respect to the original. In the scaling limit as $\eps\rightarrow 0$, this change in measure produces a drift in the limiting diffusion, and we recognize this as a Girsanov-type phenomenon. We now take a moment to explore this interpretation in a bit more detail.

The standard finite-dimensional model for Girsanov's theorem, as given at the beginning of \cite[Section 3.5]{KaratzasShreve}, is as follows. With slightly loose notation, we let $N(0,\mathbb{I}_n)$ denote the centered (multivariate) normal distribution on $\bR^n$ with covariance structure given by the identity matrix (that is, the $n$ Euclidean coordinates are i.i.d.\ normals with expectation 0 and variance 1). Let $Z$ be a random variable (on some probability space with probability denoted $P$) with distribution $N(0,\mathbb{I}_n)$, and let $v\in\bR^n$. We have a new probability measure $\tilde{P}$, absolutely continuous with respect to $P$, given by
\[
\tilde{P}(d\lambda) = e^{\ip{v}{Z(\lambda)}-\frac{1}{2}\ip{v}{v}} P(d\lambda) ,
\]
where $\ip{\cdot}{\cdot}$ is the standard inner product on $\bR^n$.
Then the random variable $Z-v$ has distribution $N(0,\mathbb{I}_n)$ under $\tilde{P}$. So adjusting the measure in this way compensates for the translation, which equivalently means that one can create a translation by adjusting the measure. The infinite-dimensional version of this (for Brownian motion on Euclidean space) is Girsanov's theorem.

Next, we rephrase this. Another way of determining $\tilde{P}$ is to say that it comes from adjusting the ``likelihood ratios'' for $P$ by
\begin{equation}\label{Eqn:Girsanov}
\frac{\tilde{P}(d\lambda_2)}{\tilde{P}(d\lambda_1)} = e^{\ip{v}{Z(\lambda_2)}-\ip{v}{Z(\lambda_1)}}
\frac{P(d\lambda_2)}{P(d\lambda_1)} .
\end{equation}
This accounts for the $ e^{\ip{v}{Z(\lambda)}}$ in the Radon-Nikodym derivative above, which is the important part; the $e^{-\frac{1}{2}\ip{v}{v}}$ is just the normalizing constant making $\tilde{P}$ a probability measure.

For the isotropic random walk, we have that $P$ is $\mu_q^0$, the uniform probability measure on the sphere of radius $\sqrt{n}$ in $T_qM$, with respect to the Riemannian inner product. (Here we choose to normalize the sphere to include the $\sqrt{n}$ factor in order to make the connection to Girsanov's theorem clearer.) Of course, $\mu_q^0$ is not a multivariate normal on $T_qM \simeq \bR^n$. However, $\mu_q^0$ has expectation $0$ and covariance matrix $\mathbb{I}_n$, so that $\mu_q^0$ has the same first two moments as  $N(0,\mathbb{I}_n)$. In light of Donsker's invariance principle, it is not surprising that ``getting the first two moments right'' is enough. Now $\mu_q^{c\eps}$ is absolutely continuous with respect to $\mu_q^0$, and, as we have seen in the proof of Theorem \ref{t:limit-Riemannian}, the relationship is given by 
\[\begin{split}
\frac{\mu_q^{c\eps}\lp d\lambda_2\rp}{\mu_q^{c\eps}\lp d\lambda_1\rp} &= \frac{\frac{1}{\mathrm{vol}(\mathbb{S}^{n-1})}\lp 1+c\eps
\ip{\grad(h)}{\lambda_2}+O\lp\eps^2\rp \rp}{\frac{1}{\mathrm{vol}(\mathbb{S}^{n-1})}\lp 1+c\eps
\ip{\grad(h)}{\lambda_1}+O\lp\eps^2\rp \rp} \cdot \frac{\mu_q^{0}\lp d\lambda_2\rp}{\mu_q^{0}\lp d\lambda_1\rp} \\
&= e^{c\eps\lp \ip{\grad(h)}{\lambda_2} - \ip{\grad(h)}{\lambda_1}\rp+O\lp \eps^2\rp}
 \cdot \frac{\mu_q^{0}\lp d\lambda_2\rp}{\mu_q^{0}\lp d\lambda_1\rp} .
\end{split}\]
Note that, as we have developed things, the random variable that has distribution $\mu_q^{0}$, which is analogous to $Z$ above, is implicitly just the identity on the sphere. (Also, $\mu_q^{c\eps}$ is a probability measure by construction, so there's no need for a normalizing factor, partially explaining our focus on the likelihood ratio.)

Comparing this to \eqref{Eqn:Girsanov}, we see that the role of $v$ is played by the quantity $c\eps\grad(h)+O(\eps^2)$. To take into account the parabolic scaling limit (and, at this stage, also to take into account the analysts' normalization), note that this non-centered measure on the sphere of radius $\sqrt{n}$ (namely $\mu_q^{c\eps}$) is mapped onto geodesics of length $\eps$, and that this step takes place in time $\dt = 2n/\eps^2$, so that the difference quotient (expected spatial displacement over elapsed time) is $2c\grad(h)+O(\eps)$. Thus, in the limit as $\eps\rightarrow 0$, we expect an infinitesimal translation given by the tangent vector $2c\grad(h)$, which is exactly what we see in Theorem \ref{t:limit-Riemannian} (appearing as a first-order differential operator). Namely, the random walk corresponding to $\mu_q^{0}$ has infinitesimal generator $\Delta_{\mathcal{R}}$ in the limit, while the random walk corresponding to $\mu_q^{c\eps}$ has infinitesimal generator $\Delta_{\mathcal{R}}+2c\grad(h)$ in the limit. So this volume sampling gives a natural random walk version of the Girsanov change of measure.

\subsection*{Acknowledgments}

This research has been partially supported by the European Research Council, ERC StG 2009 ``GeCoMethods'', contract n.\ 239748, by the ERC POC project ARTIV1 contract number 727283, by the ANR project ``SRGI'' ANR-15-CE40-0018, by a public grant as part of the Investissement d'avenir project, reference ANR-11-LABX-0056-LMH, LabEx LMH (in a joint call with Programme Gaspard Monge en Optimisation et Recherche Op\'erationnelle), by the iCODE institute, research project of the Idex Paris-Saclay, and by the SMAI project ``BOUM''. Project sponsored by the National Security Agency under Grant Number H98230-15-1-0171. The United States Government is authorized to reproduce and distribute reprints notwithstanding any copyright notation herein.

\medskip

The authors wish to thank J.-M. Bismuth for helpful discussions.

\bibliographystyle{abbrv}
\bibliography{volumesampling}

\end{document}